\documentclass[12pt,a4paper,reqno]{amsart}
\usepackage{amsfonts}
\usepackage{booktabs}
\usepackage{amsmath,amssymb,amsthm,amsxtra}
\usepackage{float}
\usepackage{amssymb}
\usepackage{mathabx}
\usepackage{bbm}
\usepackage[colorlinks, linkcolor=purple,anchorcolor=Periwinkle,
citecolor=Red,urlcolor=Emerald]{hyperref}
\usepackage[usenames,dvipsnames]{xcolor}
\usepackage{enumitem}
\usepackage{geometry,array} 
\usepackage{fullpage}
\usepackage{graphicx}
\usepackage{subfigure}
\usepackage{bookmark}
\usepackage{tikz}\usetikzlibrary{matrix}
\usepackage{url}
\usepackage{dsfont}
\usepackage[colorinlistoftodos]{todonotes}
\usepackage{tablists} \restorelistitem

\usetikzlibrary{decorations.markings}
\tikzset{->-/.style={decoration={ markings,  mark=at position #1 with
    {\arrow{>}}},postaction={decorate}}}
\tikzset{-<-/.style={decoration={ markings,  mark=at position #1 with
    {\arrow{<}}},postaction={decorate}}}
\usepackage{extarrows}
\usepackage[all]{xy}
\usepackage{setspace}
\usepackage{thmtools}
\usepackage{thm-restate}
\usepackage{hyperref}
\usepackage{cleveref}
\usepackage{multirow}
\usepackage{mathtools}
\usepackage{autobreak}
\usetikzlibrary{positioning}
\usepackage{xcolor}
\usepackage{graphicx}
\usetikzlibrary{decorations.pathreplacing}
\usetikzlibrary{calc}

\newcommand{\mfe}{\mathbf{e}}

\newcommand{\mfw}{\mathbf{w}}
\newcommand{\mfx}{\mathbf{x}}

\newcommand{\mcE}{\mathcal{E}}

\newcommand{\mcH}{\mathcal{H}}

\newcommand{\mcL}{\mathcal{L}}
\newcommand{\mcM}{\mathcal{M}}

\newcommand{\mbN}{\mathbb{N}}

\newcommand{\mbR}{\mathbb{R}}

\newcommand{\mbT}{\mathbb{T}}

\newcommand{\mbZ}{\mathbb{Z}}

\theoremstyle{plain}
\newtheorem{theorem}{Theorem}[section]

\newtheorem{lemma}[theorem]{Lemma}
\newtheorem{corollary}[theorem]{Corollary}
\newtheorem{proposition}[theorem]{Proposition}
\newtheorem{conjecture}[theorem]{Conjecture}
\theoremstyle{definition}
\newtheorem{definition}[theorem]{Definition}

\newtheorem{example}[theorem]{Example}

\newtheorem{remark}[theorem]{Remark}

\newtheorem{notations}[theorem]{Notations}
\numberwithin{equation}{section}
\newtheorem{definition-proposition}[theorem]{Definition-Proposition}

\newtheorem{assumption}[theorem]{Assumption}

\newtheorem{observation}[theorem]{Observation}
\newenvironment{definitionx}
  {
   \pushQED{\qed}\begin{definition}}
  {\popQED\end{definition}}

\newenvironment{remarkx}
  {
   \pushQED{\qed}\begin{remark}}
  {\popQED\end{remark}}

\begin{document}

\title{The logarithmic asymptotic phenomenon for generalized Markov-Hurwitz equations}

\date{\today}

\author{Zhichao Chen}
\address{School of Mathematical Sciences\\ University of Science and Technology of China \\ Hefei, Anhui 230026, P. R. China.}
\email{czc98@mail.ustc.edu.cn}

\author{Zelin Jia}
\address{Graduate School of Mathematics\\ Nagoya University\\Chikusa-ku\\ Nagoya\\464-8601\\ Japan.}
\email{zelin.jia.c0@math.nagoya-u.ac.jp}

\author{Wenchao Wu}
\address{School of Mathematical Sciences\\ University of Science and Technology of China \\ Hefei, Anhui 230026, P. R. China.}
\email{wuwch20@mail.ustc.edu.cn}

\begin{abstract} 
  The purpose of this paper is twofold. First, we introduce a family of generalized Markov-Hurwitz equations, extending classical Markov-Hurwitz equations \cite{Hur07} with additional degree $n-1$ interaction terms, Gyoda and Matsushita's generalized Markov equations \cite{GM23a} from $3$ variables to $n$ variables. 
  Second, we prove a logarithmic asymptotic phenomenon for the positive integer solutions of these equations, which may be viewed as a generalization of the results in \cite{CJ25}.
\end{abstract}
\maketitle
\tableofcontents
\newpage
\section{Introduction \& Preliminaries}\label{sec:intro-prelim}

\subsection{Introduction}

The classical Markov equation \cite{Mar80}
\[
  X_1^2+X_2^2+X_3^2=3X_1X_2X_3
\]
is one of the basic examples in which Diophantine equations, cluster algebras, and tree-like dynamics meet. 
Its higher-dimensional analogue is the Markov-Hurwitz equation \cite{Hur07}
\[
  X_1^2+\cdots+X_n^2=nX_1\cdots X_n,
\]
whose positive integer solutions are generated from the fundamental solution $(1,\dots,1)$ by Vieta involutions. The asymptotic growth of these solutions and of the associated mutation trees has been studied from several viewpoints; see, for instance, \cite{Bar94,GMR19,SV17}.

The purpose of this paper is to introduce and study the following family of generalized Markov-Hurwitz equations:
\[
  \sum_{i=1}^n X_i^2+\sum_{i=1}^n \lambda_i
  X_1\cdots \widehat{X_i}\cdots X_n
  =
  \left(n+\sum_{i=1}^n\lambda_i\right)\prod_{i=1}^n X_i,
\]
where $\lambda_i\in\mbZ_{\geq 0}$. 

In Section 2,
we define mutations for the above equation by replacing one coordinate with the other Vieta root. 
We prove that every positive integer solution is obtained from $(1,\dots,1)$ by a finite sequence of such mutations.

In Section 3 and Section 4,
we study the asymptotic behavior of this tree. 
Our observation is that, after taking logarithms, the mutation rule for generalized Markov-Hurwitz solutions becomes asymptotically additive.

More precisely, we compare the Euclid tree with the logarithmic generalized Markov-Hurwitz tree.
We show that the logarithmic generalized Markov-Hurwitz chain is asymptotic, up to a scalar multiple, to the corresponding classical $n$-branched Euclid chain. 
This generalizes the logarithmic asymptotic phenomenon for the three-variable case studied in \cite{CJ25}.

Historically, this connection goes back to the study of the growth of Markov-Hurwitz solutions. 
Zagier related the counting problem of Markov tree to an asymptotic counting problem for Euclid tree in his study of Markov numbers \cite{Zag82}, and Baragar extended this point of view to Markov-Hurwitz equations, where multi-branched Euclid trees again, play a central role in describing asymptotic growth \cite{Bar94}. 

In this and the previous paper \cite{CJ25}, still relating to the Euclid tree, we give another point of view by comparing logarithmic generalized Markov-Hurwitz tree directly with deformed Euclid tree at infinity.

\subsection{Preliminaries \& Notations}

Throughout the paper, we assume $n\geq 3$ and unless otherwise specified, a solution of a Diophantine equation means a positive integer solution. 

We denote by $\mbT_n$ the infinite rooted tree whose root has $n$ children, and in which every other node has one parent and $n-1$ children.
For instance, the following diagram shows the first two levels of $\mbT_4$; the dotted edges indicate that the same branching process continues indefinitely.

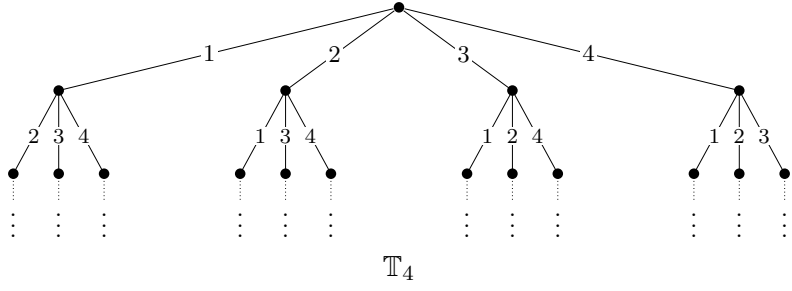
\begin{center}
\begin{tikzpicture}[
  dot/.style={circle, fill=black, inner sep=1.4pt},
  cont/.style={draw=none, fill=none, font=\scriptsize, inner sep=0pt},
  branchlabel/.style={fill=white, inner sep=0.8pt, font=\scriptsize},
  edge/.style={line width=0.35pt}
]
  \node[dot] (r) at (0,0) {};
  \foreach \x/\i/\a/\b/\c in {-4.5/1/2/3/4,-1.5/2/1/3/4,1.5/3/1/2/4,4.5/4/1/2/3} {
    \coordinate (v\i) at (\x,-1.1);
    \node[dot] at (v\i) {};
    \draw[edge] (r) -- node[pos=0.55, branchlabel] {$\i$} (v\i);
    \foreach \dx/\j/\ell in {-0.6/1/\a,0/2/\b,0.6/3/\c} {
      \coordinate (v\i\j) at ($(v\i)+(\dx,-1.1)$);
      \node[dot] at (v\i\j) {};
      \draw[edge] (v\i) -- node[pos=0.55, branchlabel, font=\tiny] {$\ell$} (v\i\j);
      \draw[edge, densely dotted] (v\i\j) -- ++(0,-0.35);
      \node[cont] at ($(v\i\j)+(0,-0.58)$) {$\vdots$};
    }
  }
  \node[cont, font=\small] at (0,-3.45) {$\mbT_4$};
\end{tikzpicture}
\end{center}

As shown in the above diagram, for the tree $\mbT_n$, we will label its branches under the following rule:

We label the branches from the root, from left to right, by $1,2,\dots,n$. At every other node, if the incoming branch has label $i$, then we label the outgoing branches, from left to right, by $1,2,\dots,\hat{i},\dots,n$.
Thus an infinite path in $\mbT_n$ can be represented by a sequence $\mfw=[w_1,w_2,\dots]$, where $1 \leq w_t \leq n$ for every $t\geq 1$, and $w_{t+1}\neq w_t$ for every $t\geq 1$.

To each node of $\mbT_n$, we associate a point in $\mbN^n$. Each branch connecting two nodes is then interpreted as a mutation between the two corresponding points.
In this paper, the mutation associated with each branch label is fixed; that is, for each label in $1,2,\dots,n$, we fix a map from $\mbN^n$ to $\mbN^n$.

\begin{definitionx}
  The \emph{classical $n$-branched Euclid tree} $\mcE_n$ is the mutation tree on $\mbT_n$ defined recursively as follows. The root is assigned the initial point $(1,1,\dots,1)\in\mbN^n$. For each label $i\in\{1,2,\dots,n\}$, define the mutation
  \begin{align*}
    \mcM_i(x_1,x_2,\dots,x_n)
    =
    (x_1,\dots,x_{i-1},x_1+\cdots+\widehat{x_i}+\cdots+x_n,x_{i+1},\dots,x_n).
  \end{align*}
  If a node is assigned the point $\mfx\in\mbN^n$, then the endpoint of its outgoing branch labeled $i$ is assigned the point $\mcM_i(\mfx)$.
\end{definitionx}

For example, the first two levels of the classical $4$-branched Euclid tree are shown in Figure~\ref{fig:classical-euclid-tree-n4}.

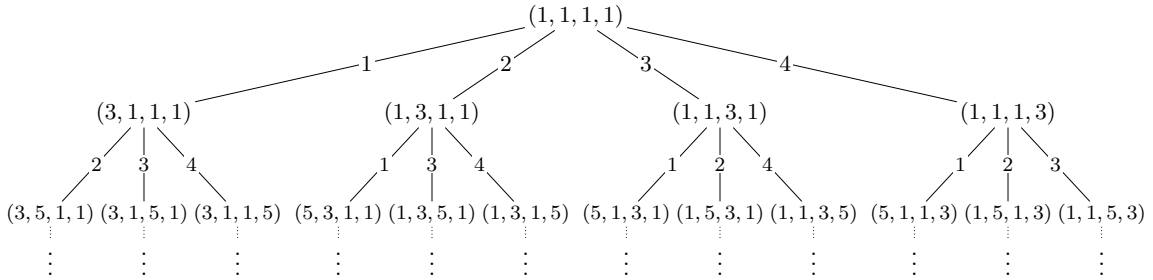
\begin{figure}[htbp]
\centering
\setcounter{figure}{-1}
\resizebox{0.95\textwidth}{!}{%
\begin{tikzpicture}[
  point/.style={draw=none, fill=none, inner sep=1pt, font=\scriptsize},
  leafpoint/.style={draw=none, fill=none, inner sep=0.8pt, font=\tiny},
  branchlabel/.style={fill=white, inner sep=0.8pt, font=\scriptsize},
  cont/.style={draw=none, fill=none, font=\scriptsize, inner sep=0pt},
  edge/.style={line width=0.35pt}
]
  \node[point] (r) at (0,0) {$(1,1,1,1)$};

  \node[point] (v1) at (-6,-1.35) {$(3,1,1,1)$};
  \node[point] (v2) at (-2,-1.35) {$(1,3,1,1)$};
  \node[point] (v3) at (2,-1.35) {$(1,1,3,1)$};
  \node[point] (v4) at (6,-1.35) {$(1,1,1,3)$};

  \draw[edge] (r) -- node[pos=0.48, branchlabel] {$1$} (v1);
  \draw[edge] (r) -- node[pos=0.48, branchlabel] {$2$} (v2);
  \draw[edge] (r) -- node[pos=0.48, branchlabel] {$3$} (v3);
  \draw[edge] (r) -- node[pos=0.48, branchlabel] {$4$} (v4);

  \node[leafpoint] (v12) at (-7.3,-2.75) {$(3,5,1,1)$};
  \node[leafpoint] (v13) at (-6,-2.75) {$(3,1,5,1)$};
  \node[leafpoint] (v14) at (-4.7,-2.75) {$(3,1,1,5)$};

  \node[leafpoint] (v21) at (-3.3,-2.75) {$(5,3,1,1)$};
  \node[leafpoint] (v23) at (-2,-2.75) {$(1,3,5,1)$};
  \node[leafpoint] (v24) at (-0.7,-2.75) {$(1,3,1,5)$};

  \node[leafpoint] (v31) at (0.7,-2.75) {$(5,1,3,1)$};
  \node[leafpoint] (v32) at (2,-2.75) {$(1,5,3,1)$};
  \node[leafpoint] (v34) at (3.3,-2.75) {$(1,1,3,5)$};

  \node[leafpoint] (v41) at (4.7,-2.75) {$(5,1,1,3)$};
  \node[leafpoint] (v42) at (6,-2.75) {$(1,5,1,3)$};
  \node[leafpoint] (v43) at (7.3,-2.75) {$(1,1,5,3)$};

  \draw[edge] (v1) -- node[pos=0.5, branchlabel, font=\tiny] {$2$} (v12);
  \draw[edge] (v1) -- node[pos=0.5, branchlabel, font=\tiny] {$3$} (v13);
  \draw[edge] (v1) -- node[pos=0.5, branchlabel, font=\tiny] {$4$} (v14);

  \draw[edge] (v2) -- node[pos=0.5, branchlabel, font=\tiny] {$1$} (v21);
  \draw[edge] (v2) -- node[pos=0.5, branchlabel, font=\tiny] {$3$} (v23);
  \draw[edge] (v2) -- node[pos=0.5, branchlabel, font=\tiny] {$4$} (v24);

  \draw[edge] (v3) -- node[pos=0.5, branchlabel, font=\tiny] {$1$} (v31);
  \draw[edge] (v3) -- node[pos=0.5, branchlabel, font=\tiny] {$2$} (v32);
  \draw[edge] (v3) -- node[pos=0.5, branchlabel, font=\tiny] {$4$} (v34);

  \draw[edge] (v4) -- node[pos=0.5, branchlabel, font=\tiny] {$1$} (v41);
  \draw[edge] (v4) -- node[pos=0.5, branchlabel, font=\tiny] {$2$} (v42);
  \draw[edge] (v4) -- node[pos=0.5, branchlabel, font=\tiny] {$3$} (v43);

  \foreach \v in {12,13,14,21,23,24,31,32,34,41,42,43} {
    \draw[edge, densely dotted] (v\v) -- ++(0,-0.35);
    \node[cont] at ($(v\v)+(0,-0.58)$) {$\vdots$};
  }

\end{tikzpicture}%
}
\caption{The first two levels of the classical $4$-branched Euclid tree $\mcE_4$.}
\label{fig:classical-euclid-tree-n4}
\end{figure}

\subsection{Main results}

Here we list all the results we have obtained in this paper.
First, let us define another tree structure associated to $\mbT_n$.

\begin{definitionx}
  Fix $\lambda=(\lambda_1,\lambda_2,\dots,\lambda_n)\in \mbZ_{\geq 0}^n$.
  The \emph{generalized Markov-Hurwitz tree} $\mcH_{n,\lambda}$ is the mutation tree on $\mbT_n$ defined recursively as follows.
  The root is assigned the initial point $(1,1,\dots,1)\in \mbN^n$.
  For each label $i\in\{1,2,\dots,n\}$, define the mutation
  \[
    \mu_i(x_1,x_2,\dots,x_n)
    =
    (x_1,\dots,x_{i-1},x_i',x_{i+1},\dots,x_n),
  \]
  where
  \[
    x_i'
    =
    \frac{\displaystyle \sum_{j\neq i}x_j^2+\lambda_i\prod_{j\neq i}x_j}{x_i}.
  \]
	Indeed, from \Cref{lem_mu}, we know that $x_i'$ will always be an integer.
  If a node is assigned the point $\mfx\in\mbN^n$, then the endpoint of its outgoing branch labeled $i$ is assigned the point $\mu_i(\mfx)$.
  Thus, all the points in this tree are integer points.
\end{definitionx}

Take $n=4$ and $\lambda=(\lambda_1,\lambda_2,\lambda_3,\lambda_4)=(0,1,2,3)$.
Then the first two levels of the generalized Markov-Hurwitz tree $\mcH_{4,\lambda}$, starting from $(1,1,1,1)$, are shown in Figure~\ref{fig:mutation-tree-n4-two-steps}.

\begin{figure}[htbp]
\centering
\resizebox{0.95\textwidth}{!}{%
\begin{tikzpicture}[
  point/.style={draw=none, fill=none, inner sep=1pt, font=\scriptsize},
  leafpoint/.style={draw=none, fill=none, inner sep=1pt, font=\tiny},
  branchlabel/.style={fill=white, inner sep=0.8pt, font=\scriptsize},
  cont/.style={draw=none, fill=none, font=\scriptsize, inner sep=0pt},
  edge/.style={line width=0.35pt}
]
  \node[point] (r) at (0,0) {$(1,1,1,1)$};

  \node[point] (v1) at (-6,-1.35) {$(3,1,1,1)$};
  \node[point] (v2) at (-2,-1.35) {$(1,4,1,1)$};
  \node[point] (v3) at (2,-1.35) {$(1,1,5,1)$};
  \node[point] (v4) at (6,-1.35) {$(1,1,1,6)$};

  \draw[edge] (r) -- node[pos=0.48, branchlabel] {$1$} (v1);
  \draw[edge] (r) -- node[pos=0.48, branchlabel] {$2$} (v2);
  \draw[edge] (r) -- node[pos=0.48, branchlabel] {$3$} (v3);
  \draw[edge] (r) -- node[pos=0.48, branchlabel] {$4$} (v4);

  \node[leafpoint] (v12) at (-7.3,-2.75) {$(3,14,1,1)$};
  \node[leafpoint] (v13) at (-6,-2.75) {$(3,1,17,1)$};
  \node[leafpoint] (v14) at (-4.7,-2.75) {$(3,1,1,20)$};

  \node[leafpoint] (v21) at (-3.3,-2.75) {$(18,4,1,1)$};
  \node[leafpoint] (v23) at (-2,-2.75) {$(1,4,26,1)$};
  \node[leafpoint] (v24) at (-0.7,-2.75) {$(1,4,1,30)$};

  \node[leafpoint] (v31) at (0.7,-2.75) {$(27,1,5,1)$};
  \node[leafpoint] (v32) at (2,-2.75) {$(1,32,5,1)$};
  \node[leafpoint] (v34) at (3.3,-2.75) {$(1,1,5,42)$};

  \node[leafpoint] (v41) at (4.7,-2.75) {$(38,1,1,6)$};
  \node[leafpoint] (v42) at (6,-2.75) {$(1,44,1,6)$};
  \node[leafpoint] (v43) at (7.3,-2.75) {$(1,1,50,6)$};

  \draw[edge] (v1) -- node[pos=0.5, branchlabel, font=\tiny] {$2$} (v12);
  \draw[edge] (v1) -- node[pos=0.5, branchlabel, font=\tiny] {$3$} (v13);
  \draw[edge] (v1) -- node[pos=0.5, branchlabel, font=\tiny] {$4$} (v14);

  \draw[edge] (v2) -- node[pos=0.5, branchlabel, font=\tiny] {$1$} (v21);
  \draw[edge] (v2) -- node[pos=0.5, branchlabel, font=\tiny] {$3$} (v23);
  \draw[edge] (v2) -- node[pos=0.5, branchlabel, font=\tiny] {$4$} (v24);

  \draw[edge] (v3) -- node[pos=0.5, branchlabel, font=\tiny] {$1$} (v31);
  \draw[edge] (v3) -- node[pos=0.5, branchlabel, font=\tiny] {$2$} (v32);
  \draw[edge] (v3) -- node[pos=0.5, branchlabel, font=\tiny] {$4$} (v34);

  \draw[edge] (v4) -- node[pos=0.5, branchlabel, font=\tiny] {$1$} (v41);
  \draw[edge] (v4) -- node[pos=0.5, branchlabel, font=\tiny] {$2$} (v42);
  \draw[edge] (v4) -- node[pos=0.5, branchlabel, font=\tiny] {$3$} (v43);

  \foreach \v in {12,13,14,21,23,24,31,32,34,41,42,43} {
    \draw[edge, densely dotted] (v\v) -- ++(0,-0.35);
    \node[cont] at ($(v\v)+(0,-0.58)$) {$\vdots$};
  }

\end{tikzpicture}%
}
\caption{The first two levels of the generalized Markov-Hurwitz tree $\mcH_{4,\lambda}$ for $\lambda=(0,1,2,3)$.}
\label{fig:mutation-tree-n4-two-steps}
\end{figure}
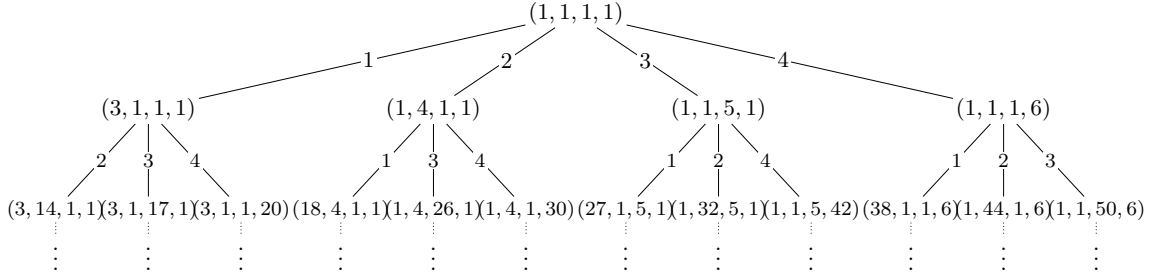

Here is our first main result:

\begin{theorem}[\Cref{thm: generate}]
  With the notation of Definition 1.2, every positive integer solution of the generalized Markov-Hurwitz equation (\ref{eq: generalized Markov Hurwitz equations}) appears as a node of the generalized Markov-Hurwitz tree $\mcH_{n,\lambda}$. More precisely, for any positive integer solution $(x_1,x_2,\dots,x_n)$, there exists a finite path $\mfw=[w_1,\dots,w_t]$ in $\mbT_n$ such that
  \[
    (x_1,x_2,\dots,x_n)
    =
    (\mu_{w_t}\circ\cdots\circ\mu_{w_1})(1,1,\dots,1).
  \]
\end{theorem}

To study the asymptotic behavior of the mutation dynamics, we apply the coordinatewise logarithm to the generalized Markov-Hurwitz tree.

\begin{definitionx}
  With the notation of Definition 1.2, define $\widebar{x}\coloneqq \log x$ for any positive real number $x$, and write
  \[
    \widebar{\mfx}
    =
    (\widebar{x_1},\widebar{x_2},\dots,\widebar{x_n})
    =
    (\log x_1,\log x_2,\dots,\log x_n)
  \]
  for $\mfx=(x_1,x_2,\dots,x_n)\in\mbN^n$.
  The \emph{logarithmic generalized Markov-Hurwitz tree} $\widebar{\mcH}_{n,\lambda}$ is the tree obtained from $\mcH_{n,\lambda}$ by replacing each node $\mfx$ by $\widebar{\mfx}$.
  
  Thus, along a path $\mfw=[w_1,w_2,\dots]$ in $\mbT_n$, the logarithmic generalized Markov-Hurwitz chain is obtained by taking the logarithm of each coordinate in the corresponding chain in $\mcH_{n,\lambda}$.
\end{definitionx}

Now, we can introduce our second main result:

\begin{theorem}[\Cref{thm: logarithmic asymptotic}]
  With the notation above, let $\mfw=[w_1,w_2,\dots]\in\mbT_n$ be a generic mutation sequence.
  Then there exists a real number $q$ such that the logarithmic generalized Markov-Hurwitz chain in $\widebar{\mcH}_{n,\lambda}$ along $\mfw$ is asymptotic to $q$ times the corresponding classical $n$-branched Euclid chain in $\mcE_n$ along $\mfw$.
\end{theorem}

The following diagram illustrates the logarithmic version of Figure~\ref{fig:mutation-tree-n4-two-steps}. Each coordinate is replaced by its natural logarithm and rounded to two decimal places.

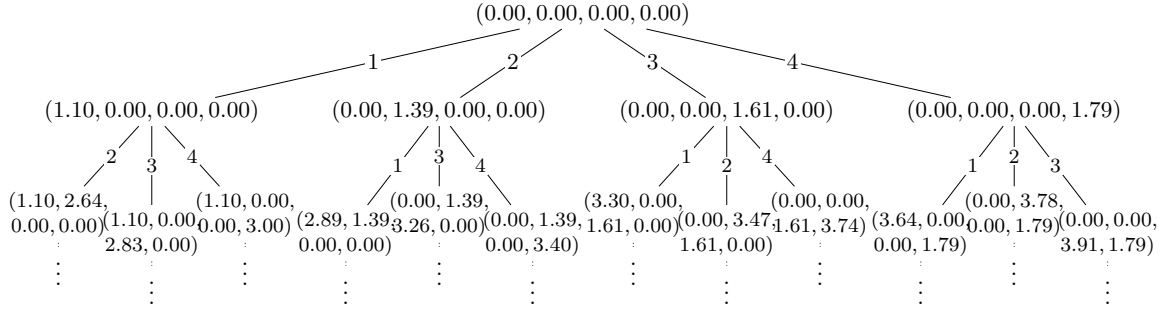
\begin{figure}[htbp]
\centering
\resizebox{0.98\textwidth}{!}{%
\begin{tikzpicture}[
  point/.style={draw=none, fill=none, inner sep=1pt, font=\scriptsize},
  leafpoint/.style={draw=none, fill=none, inner sep=1pt, font=\tiny},
  branchlabel/.style={fill=white, inner sep=0.8pt, font=\scriptsize},
  cont/.style={draw=none, fill=none, font=\scriptsize, inner sep=0pt},
  edge/.style={line width=0.35pt}
]
  \node[point] (r) at (0,0) {$(0.00,0.00,0.00,0.00)$};

  \node[point] (v1) at (-6,-1.35) {$(1.10,0.00,0.00,0.00)$};
  \node[point] (v2) at (-2,-1.35) {$(0.00,1.39,0.00,0.00)$};
  \node[point] (v3) at (2,-1.35) {$(0.00,0.00,1.61,0.00)$};
  \node[point] (v4) at (6,-1.35) {$(0.00,0.00,0.00,1.79)$};

  \draw[edge] (r) -- node[pos=0.48, branchlabel] {$1$} (v1);
  \draw[edge] (r) -- node[pos=0.48, branchlabel] {$2$} (v2);
  \draw[edge] (r) -- node[pos=0.48, branchlabel] {$3$} (v3);
  \draw[edge] (r) -- node[pos=0.48, branchlabel] {$4$} (v4);

  \node[leafpoint] (v12) at (-7.3,-2.78) {$\begin{array}{c}(1.10,2.64,\\0.00,0.00)\end{array}$};
  \node[leafpoint] (v13) at (-6,-3.04) {$\begin{array}{c}(1.10,0.00,\\2.83,0.00)\end{array}$};
  \node[leafpoint] (v14) at (-4.7,-2.78) {$\begin{array}{c}(1.10,0.00,\\0.00,3.00)\end{array}$};

  \node[leafpoint] (v21) at (-3.3,-3.04) {$\begin{array}{c}(2.89,1.39,\\0.00,0.00)\end{array}$};
  \node[leafpoint] (v23) at (-2,-2.78) {$\begin{array}{c}(0.00,1.39,\\3.26,0.00)\end{array}$};
  \node[leafpoint] (v24) at (-0.7,-3.04) {$\begin{array}{c}(0.00,1.39,\\0.00,3.40)\end{array}$};

  \node[leafpoint] (v31) at (0.7,-2.78) {$\begin{array}{c}(3.30,0.00,\\1.61,0.00)\end{array}$};
  \node[leafpoint] (v32) at (2,-3.04) {$\begin{array}{c}(0.00,3.47,\\1.61,0.00)\end{array}$};
  \node[leafpoint] (v34) at (3.3,-2.78) {$\begin{array}{c}(0.00,0.00,\\1.61,3.74)\end{array}$};

  \node[leafpoint] (v41) at (4.7,-3.04) {$\begin{array}{c}(3.64,0.00,\\0.00,1.79)\end{array}$};
  \node[leafpoint] (v42) at (6,-2.78) {$\begin{array}{c}(0.00,3.78,\\0.00,1.79)\end{array}$};
  \node[leafpoint] (v43) at (7.3,-3.04) {$\begin{array}{c}(0.00,0.00,\\3.91,1.79)\end{array}$};

  \draw[edge] (v1) -- node[pos=0.5, branchlabel, font=\tiny] {$2$} (v12);
  \draw[edge] (v1) -- node[pos=0.5, branchlabel, font=\tiny] {$3$} (v13);
  \draw[edge] (v1) -- node[pos=0.5, branchlabel, font=\tiny] {$4$} (v14);

  \draw[edge] (v2) -- node[pos=0.5, branchlabel, font=\tiny] {$1$} (v21);
  \draw[edge] (v2) -- node[pos=0.5, branchlabel, font=\tiny] {$3$} (v23);
  \draw[edge] (v2) -- node[pos=0.5, branchlabel, font=\tiny] {$4$} (v24);

  \draw[edge] (v3) -- node[pos=0.5, branchlabel, font=\tiny] {$1$} (v31);
  \draw[edge] (v3) -- node[pos=0.5, branchlabel, font=\tiny] {$2$} (v32);
  \draw[edge] (v3) -- node[pos=0.5, branchlabel, font=\tiny] {$4$} (v34);

  \draw[edge] (v4) -- node[pos=0.5, branchlabel, font=\tiny] {$1$} (v41);
  \draw[edge] (v4) -- node[pos=0.5, branchlabel, font=\tiny] {$2$} (v42);
  \draw[edge] (v4) -- node[pos=0.5, branchlabel, font=\tiny] {$3$} (v43);

  \foreach \v in {12,13,14,21,23,24,31,32,34,41,42,43} {
    \draw[edge, densely dotted] (v\v) -- ++(0,-0.45);
    \node[cont] at ($(v\v)+(0,-0.72)$) {$\vdots$};
  }

\end{tikzpicture}%
}
\caption{The first two levels of the logarithmic generalized Markov-Hurwitz tree $\widebar{\mcH}_{4,\lambda}$ for $\lambda=(0,1,2,3)$, with entries rounded to two decimal places.}
\label{fig:log-mutation-tree-n4-two-steps}
\end{figure}

The next diagram compares Figure~\ref{fig:log-mutation-tree-n4-two-steps} with Figure~\ref{fig:classical-euclid-tree-n4} coordinatewise. Namely, if corresponding nodes in Figure~\ref{fig:log-mutation-tree-n4-two-steps} and Figure~\ref{fig:classical-euclid-tree-n4} are labeled by $(\widebar{x_1},\widebar{x_2},\widebar{x_3},\widebar{x_4})$ and $(e_1,e_2,e_3,e_4)$ respectively, then the corresponding node below is labeled by
\[
  \frac{\widebar{\mfx}}{\mfe}
  \coloneqq
  \left(\frac{\widebar{x_1}}{e_1},\frac{\widebar{x_2}}{e_2},\frac{\widebar{x_3}}{e_3},\frac{\widebar{x_4}}{e_4}\right),
\]
again rounded to two decimal places.

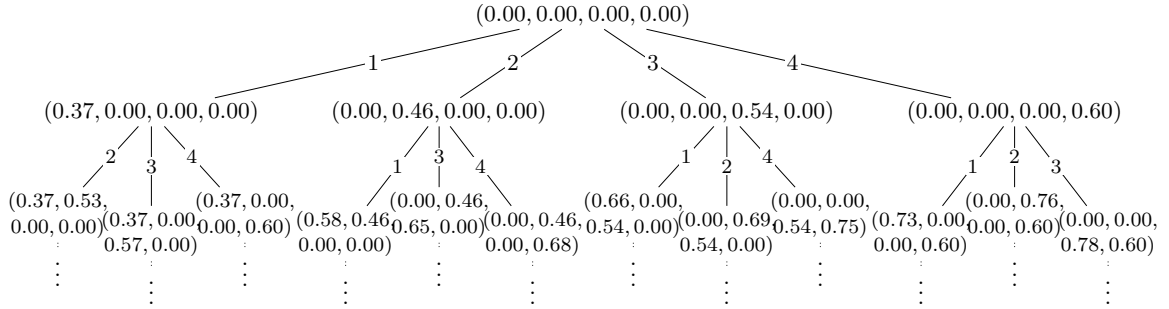
\begin{figure}[htbp]
\centering
\resizebox{0.98\textwidth}{!}{%
\begin{tikzpicture}[
  point/.style={draw=none, fill=none, inner sep=1pt, font=\scriptsize},
  leafpoint/.style={draw=none, fill=none, inner sep=0.8pt, font=\tiny},
  branchlabel/.style={fill=white, inner sep=0.8pt, font=\scriptsize},
  cont/.style={draw=none, fill=none, font=\scriptsize, inner sep=0pt},
  edge/.style={line width=0.35pt}
]
  \node[point] (r) at (0,0) {$(0.00,0.00,0.00,0.00)$};

  \node[point] (v1) at (-6,-1.35) {$(0.37,0.00,0.00,0.00)$};
  \node[point] (v2) at (-2,-1.35) {$(0.00,0.46,0.00,0.00)$};
  \node[point] (v3) at (2,-1.35) {$(0.00,0.00,0.54,0.00)$};
  \node[point] (v4) at (6,-1.35) {$(0.00,0.00,0.00,0.60)$};

  \draw[edge] (r) -- node[pos=0.48, branchlabel] {$1$} (v1);
  \draw[edge] (r) -- node[pos=0.48, branchlabel] {$2$} (v2);
  \draw[edge] (r) -- node[pos=0.48, branchlabel] {$3$} (v3);
  \draw[edge] (r) -- node[pos=0.48, branchlabel] {$4$} (v4);

  \node[leafpoint] (v12) at (-7.3,-2.78) {$\begin{array}{c}(0.37,0.53,\\0.00,0.00)\end{array}$};
  \node[leafpoint] (v13) at (-6,-3.04) {$\begin{array}{c}(0.37,0.00,\\0.57,0.00)\end{array}$};
  \node[leafpoint] (v14) at (-4.7,-2.78) {$\begin{array}{c}(0.37,0.00,\\0.00,0.60)\end{array}$};

  \node[leafpoint] (v21) at (-3.3,-3.04) {$\begin{array}{c}(0.58,0.46,\\0.00,0.00)\end{array}$};
  \node[leafpoint] (v23) at (-2,-2.78) {$\begin{array}{c}(0.00,0.46,\\0.65,0.00)\end{array}$};
  \node[leafpoint] (v24) at (-0.7,-3.04) {$\begin{array}{c}(0.00,0.46,\\0.00,0.68)\end{array}$};

  \node[leafpoint] (v31) at (0.7,-2.78) {$\begin{array}{c}(0.66,0.00,\\0.54,0.00)\end{array}$};
  \node[leafpoint] (v32) at (2,-3.04) {$\begin{array}{c}(0.00,0.69,\\0.54,0.00)\end{array}$};
  \node[leafpoint] (v34) at (3.3,-2.78) {$\begin{array}{c}(0.00,0.00,\\0.54,0.75)\end{array}$};

  \node[leafpoint] (v41) at (4.7,-3.04) {$\begin{array}{c}(0.73,0.00,\\0.00,0.60)\end{array}$};
  \node[leafpoint] (v42) at (6,-2.78) {$\begin{array}{c}(0.00,0.76,\\0.00,0.60)\end{array}$};
  \node[leafpoint] (v43) at (7.3,-3.04) {$\begin{array}{c}(0.00,0.00,\\0.78,0.60)\end{array}$};

  \draw[edge] (v1) -- node[pos=0.5, branchlabel, font=\tiny] {$2$} (v12);
  \draw[edge] (v1) -- node[pos=0.5, branchlabel, font=\tiny] {$3$} (v13);
  \draw[edge] (v1) -- node[pos=0.5, branchlabel, font=\tiny] {$4$} (v14);

  \draw[edge] (v2) -- node[pos=0.5, branchlabel, font=\tiny] {$1$} (v21);
  \draw[edge] (v2) -- node[pos=0.5, branchlabel, font=\tiny] {$3$} (v23);
  \draw[edge] (v2) -- node[pos=0.5, branchlabel, font=\tiny] {$4$} (v24);

  \draw[edge] (v3) -- node[pos=0.5, branchlabel, font=\tiny] {$1$} (v31);
  \draw[edge] (v3) -- node[pos=0.5, branchlabel, font=\tiny] {$2$} (v32);
  \draw[edge] (v3) -- node[pos=0.5, branchlabel, font=\tiny] {$4$} (v34);

  \draw[edge] (v4) -- node[pos=0.5, branchlabel, font=\tiny] {$1$} (v41);
  \draw[edge] (v4) -- node[pos=0.5, branchlabel, font=\tiny] {$2$} (v42);
  \draw[edge] (v4) -- node[pos=0.5, branchlabel, font=\tiny] {$3$} (v43);

  \foreach \v in {12,13,14,21,23,24,31,32,34,41,42,43} {
    \draw[edge, densely dotted] (v\v) -- ++(0,-0.45);
    \node[cont] at ($(v\v)+(0,-0.72)$) {$\vdots$};
  }

\end{tikzpicture}%
}
\caption{The coordinatewise quotient of Figure~\ref{fig:log-mutation-tree-n4-two-steps} by Figure~\ref{fig:classical-euclid-tree-n4}, with entries rounded to two decimal places.}
\label{fig:quotient-log-mutation-tree-n4-two-steps}
\end{figure}

In this example, our second main result says that for any infinite path $\mfw\in\mbT_4$ in which each label $1,2,3,4$ appears infinitely many times, there exists a real number $q\in\mbR$ such that the corresponding point at infinity is $(q,q,q,q)$.

Strikingly, even after only two levels, the figure already hints at this phenomenon in each point: the newly mutated coordinates are becoming close to one another.

\section{Generalized Markov-Hurwitz equations}
In this section, we deal with a class of new Diophantine equations, which is a generalization of Markov-Hurwitz equations. The \textit{generalized Markov-Hurwitz equations} are given as follows,

\begin{align}\label{eq: generalized Markov Hurwitz equations}
 \displaystyle \sum_{i=1}^n X_i^2 + \sum_{i=1}^n \lambda_i X_1\cdots \widehat{X_i} \cdots X_n = (n + \sum_{i=1}^n \lambda_i) \prod_{i=1}^n X_i,
\end{align}
where $\lambda_i \in \mbZ_{\geq 0}$ for all $i$, and by the symbol $\widehat{X_i}$ we mean that the factor $X_i$ is omitted from the product.

\begin{remarkx}\label{rmk: generalized cluster algebra}\
We have two comments about the generalized Markov-Hurwitz equations.
	\begin{enumerate}
		\item  In the case of $\lambda_i=0$ for every $1\leq i \leq n$, Equation (\ref{eq: generalized Markov Hurwitz equations}) becomes the classical Markov-Hurwitz equation \cite{Hur07}.
		\item  In the case of $n=3$, Equation (\ref{eq: generalized Markov Hurwitz equations}) becomes the generalized Markov equation introduced by Gyoda and Matsushita \cite{GM23a}.
	\end{enumerate}
	\vspace{-20pt}
\end{remarkx}

\begin{lemma}\label{lem_mu}
	Let $(x_1, x_2, \ldots, x_n)$ be a positive integer solution to Equation (\ref{eq: generalized Markov Hurwitz equations}) and fix an index $m$ satisfying $1\leq m \leq n$. 
	Set
	\[
	x_m'
	:=
	\frac{\left(x_1^2+\cdots+x_{m-1}^2+x_{m+1}^2+\cdots+x_n^2\right)
		+\lambda_m x_1\cdots \widehat{x_m}\cdots x_n}{x_m},
	\]
	then $(x_1,\dots,x_{m-1},x_m',x_{m+1},\dots,x_n)$ is also a positive integer solution to Equation (\ref{eq: generalized Markov Hurwitz equations}).
\end{lemma}

\begin{proof}
	Set $p_m:=x_1\cdots \widehat{x_m}\cdots x_n$ and $q_m:=x_1^2+\cdots+x_{m-1}^2+x_{m+1}^2+\cdots+x_n^2+\lambda_m p_m$.
    Substituting \(X_i=x_i\) for all \(i\ne m\) into Equation~(\ref{eq: generalized Markov Hurwitz equations}), we obtain the following quadratic equation in \(X_m\):
	\[
	X_m^2
	-\left(n+\sum_{i=1}^n\lambda_i-\sum_{i\neq m}\frac{\lambda_i}{x_i}\right)p_mX_m + q_m
	=0.
	\]
	
 By Vieta's formulas, the product of the two roots is \(q_m\). 	Obviously, $x_m$ is a root of this quadratic equation. Since \(x_mx_m'=q_m\), we conclude that \(x_m'\) is the other root of this quadratic equation.
 Equivalently, $(x_1,\dots,x_{m-1},x_m',x_{m+1},\dots,x_n)$ is again a positive solution of Equation (\ref{eq: generalized Markov Hurwitz equations}). Since $x_m+x_m' = \left(n+\sum_{i=1}^n\lambda_i-\sum_{i\neq m}\frac{\lambda_i}{x_i}\right)p_m$, $x_m'$ is an integer.
\end{proof}

Now we define a mutation rules associated to the set of positive integer solutions of Equation (\ref{eq: generalized Markov Hurwitz equations}).

\begin{definition}\label{mutation rules}
	With the above notion, the \textit{solution mutation $\mu_i$ in direction $i$} is defined as follows
	\[
	\mu_i(x_1, \dots, x_i, \dots, x_n) \notag = (x_1, \dots, x_i', \dots, x_n).
	\]
\end{definition}

\begin{remarkx}\label{rmk: generalized mutations}
In the case \(n=3\), our mutation rules coincide with those defined in \cite{GM23a}.
\end{remarkx}
From the definition, it follows that applying the mutation in the same direction twice returns the original point. That is,
\begin{align}\label{Vieta jumping}
	\mu_i^2(x_1,x_2,\ldots,x_n)=(x_1,x_2,\ldots,x_n).
\end{align}
We call this the \emph{Vieta jumping phenomenon}.


The proof of the following lemma is an analogue and generalization of \cite[Proposition 18]{GMR19}.

\begin{lemma}\label{lem_posi}
	Let $x:= (x_1, x_2, \dots, x_n)$ be a positive solution to Equation (\ref{eq: generalized Markov Hurwitz equations}), distinct from $(1,1,\dots,1)$. 
  If $x_j$ is the largest coordinate of $x$, then the largest entry of $\mu_j(x)$ is strictly smaller than $x_j$, that is, $(\mu_j(x))_i<x_j$ for all $i$.
\end{lemma}

\begin{proof}
	After simultaneously relabeling $x_i$ and $\lambda_i$, we may assume that $x_1\le x_2\le \cdots \le x_n$.
	Similar to Lemma \ref{lem_mu}, set $p \coloneqq x_1\cdots x_{n-1}$ and $q \coloneqq \sum_{i=1}^{n-1}x_i^2+\lambda_n p$.
	Substituting \(X_i=x_i\) for all \(i\ne n\) into Equation~(\ref{eq: generalized Markov Hurwitz equations}), we obtain the following quadratic polynomial in $T$:
	\[
	f(T) = T^2
	-\left(n+\sum_{i=1}^n\lambda_i-\sum_{i=1}^{n-1}\frac{\lambda_i}{x_i}\right)pT + q.
	\]
	Then $f(T) = 0$ has roots at $x_n$ and $x_n'$, where $x_n'$ is the newly mutated coordinate in $\mu_n(x)$.
	It is easy to see that the lemma holds unless we have either
	\[
    x_{n-1}\le x_n\le x_n'
	\]
	or
	\[
	  x_n'<x_{n-1}=x_n.
	\]

	In both cases, since the coefficient of $T^2$ is positive, it follows that
	$f(x_{n-1}) \ge 0$. 
  A direct computation gives
	\begin{align*}
		f(x_{n-1})
		&=
		\sum_{i=1}^{n-2}x_i^2+2x_{n-1}^2
		+(\lambda_{n-1}+\lambda_n)p
		-(n+\sum_{i=1}^n\lambda_i)px_{n-1}
		+\sum_{i=1}^{n-2}\frac{\lambda_i}{x_i}px_{n-1}.
	\end{align*}
	Since $1\le x_1\leq \cdots \leq x_{n-1}$, we have the following inequalities:
	\[
	  \sum_{i=1}^{n-2}x_i^2\leq (n-2)x_{n-1}^2, \quad p\leq px_{n-1}, \quad \frac{p}{x_i}x_{n-1}\leq px_{n-1}.
	\]
	Therefore, we obtain
	\begin{align*}
		f(x_{n-1})
		&\leq
		(n-2)x_{n-1}^2+2x_{n-1}^2
		+(\lambda_{n-1}+\lambda_n)px_{n-1} 
		-(n+\sum_{i=1}^n\lambda_i)px_{n-1}
		+\sum_{i=1}^{n-2}\lambda_i px_{n-1} \\
		&= n(1-\sum_{i=1}^{n-2}x_i)x_{n-1}^2\leq 0.
	\end{align*}
	Hence $f(x_{n-1}) = 0$, so we have
	\[
	  x_1=x_2=\cdots=x_{n-2}=1.
	\]
	
	Thus, the generalized Markov-Hurwitz equation reduces to
	\begin{align}\label{eq: reduced Markov Hurwitz equations}
	  x_{n-1}^2+x_n^2+\lambda_{n-1}x_n+\lambda_n x_{n-1}+(n-2) = (n+\lambda_{n-1}+\lambda_n)x_{n-1}x_n.
  \end{align}
	
	As we mentioned before, this lemma holds unless we have either $x_{n-1}\le x_n\le x_n'$ or $x_n'<x_{n-1}=x_n$, now we will study these two cases separately.
	
	\noindent \emph{Case 1:} \(x_n=x_{n-1}\).
	
	Then Equation (\ref{eq: reduced Markov Hurwitz equations}) becomes
	\[
	  (n+\lambda_{n-1}+\lambda_n-2)x_n^2-(\lambda_{n-1}+\lambda_n)x_n-(n-2)=0.
	\]
	This factors as
	\[
	  (x_n-1)\bigl((n+\lambda_{n-1}+\lambda_n-2)x_n+(n-2)\bigr)=0.
	\]
	Since \(x_n>0\) and $n\geq 2$, it follows that $x_n=x_{n-1}=1$.
	Hence in this case, we finally obtain
	\[
	  x=(1, 1, \dots, 1).
	\]
  Which is a contradiction to the assumption that $x$ is distinct from $(1,1,\dots,1)$.
	
	\noindent \emph{Case 2:} \(x_n'\geq x_n>x_{n-1}\).

	Set $s:=n+\lambda_{n-1}+\lambda_n$.
	Consider the following quadratic polynomial
	\[
	g(T)=T^2-(sx_{n-1}-\lambda_{n-1})T+(x_{n-1}^2+\lambda_n x_{n-1}+n-2).
	\]
	which is obtained by replacing $x_n$ by $T$ in Equation (\ref{eq: reduced Markov Hurwitz equations}).
	Since $x_n$ is a root of $g$, the other root is $x_n':=sx_{n-1}-\lambda_{n-1}-x_n$.
	Moreover, by substituting $T$ with $x_{n-1}$, we have
	\begin{align*}
		g(x_{n-1})
		&=x_{n-1}^2-(sx_{n-1}-\lambda_{n-1})x_{n-1}+(x_{n-1}^2+\lambda_n x_{n-1}+n-2) \\
		&=-(x_{n-1}-1)\bigl((s-2)x_{n-1}+n-2\bigr).
	\end{align*}
	Since $x_{n-1}\ge 1$ and $s\geq n\geq 2$, it follows that $g(x_{n-1})\leq0$.
	Therefore $x_{n-1}$ sits between the two roots of $g$, that is, $x_n'\geq x_{n-1}\geq x_n$.
  However, this contradicts to the assumption that $x_n>x_{n-1}$.

  Hence, both cases lead to contradictions, so the lemma holds.
\end{proof}

\begin{theorem}\label{thm: generate}
	Every positive integer solution of the generalized Markov-Hurwitz equation (\ref{eq: generalized Markov Hurwitz equations}) appears as a node of the generalized Markov-Hurwitz tree $\mcH_{n,\lambda}$. More precisely, for any positive integer solution $(x_1,x_2,\dots,x_n)$, there exists a finite path $\mfw=[w_1,\dots,w_t]$ in $\mbT_n$ such that
  \[
    (x_1,x_2,\dots,x_n)
    =
    (\mu_{w_t}\circ\cdots\circ\mu_{w_1})(1,1,\dots,1).
  \]
\end{theorem}

\begin{proof}
  Suppose that $(x_1,x_2,\dots,x_n)$ is any positive integer solution to  Equation (\ref{eq: generalized Markov Hurwitz equations}) which is distinct from $(1,1,\dots,1)$.
  
  By Lemma \ref{lem_posi},  if we repeatedly mutate in the direction of a largest coordinate
  (note that the largest coordinate is unique from \Cref{lem_posi}), then after finitely many steps,  we eventually reach the initial solution
  \((1,1,\dots,1)\). Due to the Vieta jumping phenomenon (\ref{Vieta jumping}), we can obtain $(x_1,x_2,\dots,x_n)$ from $(1,1,\dots,1)$.
  Hence, $(x_1,x_2,\dots,x_n)$ can be generated from $(1,1,\dots,1)$ by finitely many mutations.
\end{proof}

We conclude this section with the following proposition.

\begin{proposition}\label{mutate to max}
  If $x_j$ is not the largest coordinate of $(x_1,x_2, \dots, x_n)$, then it becomes the largest after the mutation $\mu_j$.
  That is, $(\mu_j(x_1, x_2, \dots, x_n))_j > (\mu_j (x_1, x_2, \dots, x_n))_i$ for all $i \neq j$. 
\end{proposition}

\begin{proof}
  Without loss of generality, we assume $x_1 \leq x_2 \leq \dots < x_n$, let $(x_1^{\prime}, x_2^{\prime}, \dots x_n^{\prime}) = \mu_j (x_1, x_2, \dots, x_n)$ with $j < n$.
  We can calculate
  \begin{align}
    x_j^{\prime} - x_n &= \frac{\left(x_1^2+\cdots+x_{j-1}^2+x_{j+1}^2+\cdots+x_n^2\right) +\lambda_j x_1\cdots \widehat{x_j}\cdots x_n}{x_j} -x_n \notag \\
    &= (n + \sum_{i \neq j} \lambda_i) \prod_{i \neq j} x_i - \sum_{i \neq j} \lambda_i (\prod_{m \neq i,j} x_m) - x_j - x_n \notag \\
    &= x_n \times [(n + \sum_{i \neq j} \lambda_i) \prod_{i \neq j,n} x_i - \sum_{i \neq j,n} \lambda_i (\prod_{m \neq i,j,n} x_m)-1] - \lambda_n \prod_{m \neq j,n}x_m - x_j \notag \\
    &> x_n \times [(n + \sum_{i \neq j} \lambda_i) \prod_{i \neq j,n} x_i - \sum_{i \neq j,n} \lambda_i (\prod_{m \neq i,j,n} x_m)- \lambda_n \prod_{m \neq j,n}x_m -1] - x_j \notag \\
    &\geq x_n \times (n - 1) -x_j > 0.
  \end{align}
  Thus, $x_j^{\prime}$ is the largest of $(x_1^{\prime}, x_2^{\prime}, \dots x_n^{\prime})$.
\end{proof}

\section{Asymptotic behavior of $k$-deformed $n$-branched Euclid trees}

In this section, we study the asymptotic phenomenon of rather simpler mutation trees, namely, the $k$-deformed $n$-branched Euclid trees.
We aim to prove the asymptotic phenomenon between the $k$-deformed $n$-branched Euclid trees and the classical $n$-branched Euclid trees.

\begin{definitionx}
  Let $n \in \mbN_{\geq 2}$ and $k \in \mbZ_{\geq 0}$.
  The \textit{$k$-deformed $n$-branched Euclid tree} $\mcE_{n,k}$ is defined as the tree generated from the initial point $(e_1,e_2,\dots,e_n)$ by the following mutation rules:
  \begin{align}
    \mcM_{i;k} (x_1, x_2, \ldots, x_n) = (x_1, x_2, \dots, x_1 + \cdots + \widehat{x_i} + \cdots + x_n + k, \dots, x_n)
  \end{align} 
  where $i=1,2,\dots,n$.
The $0$-deformed $n$-branched Euclid tree will be called the \emph{classical $n$-branched Euclid tree}, and will be denoted simply by $\mathcal E_n$.
\end{definitionx}

The name ``$n$-branched Euclid tree'' first appears in \cite{Bar94}, where the author studied the asymptotic growth of those trees.

\begin{remarkx}
Note that in the previous section, we define a tree by the initial solution point $(1,1,\dots,1)$ and solution mutations associate to Equation (\ref{eq: generalized Markov Hurwitz equations}).
However, in the current section, we can consider the tree generated from any initial point $(e_1,e_2,\dots,e_n)$, where $e_i \in \mbZ_{\geq 0}$ for $i=1,2,\dots,n$.
\end{remarkx}

To clarify the statements and proofs in this section, we need to introduce some notations:

\begin{notations}\label{Notation for Euclid}
  For the $k$-deformed $n$-branched Euclid tree $\mcE_{n,k}$ and any initial point 
  \[
    (e_1,e_2,\dots,e_n) \in \mbZ_{\geq 0}^n
  \]
  take $\mfw = [w_1, w_2, \dots] \in \mbT_n$.

  We then obtain a sequence of points by doing the mutation chain following the directions in $\mfw$ which will be denoted by the following:
  \begin{align*}
    (e_1,e_2,\dots,e_n) \xrightarrow{\mcM_{w_1;k}} ({}_1 e_1, {}_1 e_2, \dots, {}_1 e_n) \xrightarrow{\mcM_{w_2;k}} ({}_2 e_1, {}_2 e_2, \dots, {}_2 e_n) \xrightarrow{\mcM_{w_3;k}} \cdots
  \end{align*}
  Note that from the definition of the mutation rules, we only change one coordinate at each mutation.
  Thus we can take the newly changed number each time after each mutation to form a sequence of numbers, which will be denoted by the following:
  \begin{align*}
    \mathfrak{e}_1 \coloneqq {}_1 e_{w_1}, \mathfrak{e}_2 \coloneqq {}_2 e_{w_2}, \dots
  \end{align*}
  In this way, we obtain a sequence of numbers $\{\mathfrak{e}_i\}_{i=1}^{\infty}$ from the mutation chain $\mcM_k^{\mfw}$ starting from the initial point $(e_1,e_2,\dots,e_n)$.
\end{notations}

Take $\mfw \in \mbT_n$, and suppose that the direction $i$ appears in $\mfw$ only finitely many times, i.e. $\mfw = [w_1, w_2, \dots, w_N=i, w_{N+1}, \dots]$, and $i \neq w_{N+j}$ for any $j \in \mbN$.
Then considering the mutation $\mcM_k^{\mfw} (e_1,e_2,\dots,e_n)$ is the same as considering the mutation 
\[
  \mcM_k^{[w_{N+1}, \dots]}(\mcM_k^{[w_1, w_2, \dots, w_{N}]}(e_1,e_2,\dots,e_n)).
\]
Denote $\mcM_k^{[w_1, w_2, \dots, w_{N}]}(e_1,e_2,\dots,e_n)$ by $(e_1^{\prime}, e_2^{\prime}, \dots, e_n^{\prime})$, take $e_i^{\prime}$ out and we rearrange the directions in the following way:

\begin{itemize}
  \item If $j < i$, then we keep the direction $j$ unchanged.
  \item If $j > i$, then we replace the direction $j$ by the direction $j-1$.
\end{itemize}

Then we denote such point after rearranging as $(e_1^{\prime \prime},e_2^{\prime \prime}, \dots, e_{n-1}^{\prime \prime})$.
Denote $\mfw^{\prime} \in \mbT_{n-1}$ the sequence obtained from $\mfw$ by cutting out $[w_1,w_2,\dots, w_N=i]$. 
Then it is easy to see that the sequence of numbers $\{\mathfrak{e}_j\}_{j=i}^{\infty}$ obtained from the mutation chain $\mcM_{k+e_i^{\prime}}^{\mfw^{\prime}} (e_1^{\prime \prime},e_2^{\prime \prime}, \dots, e_{n-1}^{\prime \prime})$ 
can be regarded the same as the sequence of numbers $\{\mathfrak{e}_j\}_{j=i}^{\infty}$ obtained from the mutation chain $\mcM_k^{[w_{N+1}, \dots]}(e_1^{\prime}, e_2^{\prime}, \dots, e_n^{\prime})$.

Thus, if some directions appears in $\mfw \in \mbT_n$ only finitely many steps, then we can reduce it to the case where $\mfw^{\prime} \in \mbT_{m}$, where $m < n$.
Therefore, from now on, we will only consider the case under the following assumption:

\begin{assumption}\label{generic assumption}
  For $\mfw \in \mbT_n$, we assume that each direction appears infinitely many times in $\mfw$, we call such $\mfw$ a generic sequence in $\mbT_n$.
\end{assumption}

\subsection{Comparison between $k$-deformed Euclid trees and classical Euclid trees}
  Let $\mcE_{n,k}$ be the $k$-deformed $n$-branched Euclid tree starting from an arbitrary initial point $(\beta_1, \beta_2, \dots, \beta_n) \in \mbZ_{\geq 0}^n$,
  and $\mcE_n$ be the classical $n$-branched Euclid tree starting from an arbitrary initial point $(\alpha_1, \alpha_2, \dots, \alpha_n) \in \mbZ_{\geq 0}^n$.

  Since both trees have the same underlying tree structure, which we denote by \(\mathbb T_n\), every position
  \(\Theta\in \mathbb T_n\) determines two corresponding points
  \[
  (y_1,y_2,\dots,y_n)\in \mathcal E_{n,k}
  \quad\text{and}\quad
  (x_1,x_2,\dots,x_n)\in \mathcal E_n.
  \]
  \begin{definitionx}[Comparison point]\label{Comparison point}
  	With the above notation, the \emph{comparison point} associated to \(\Theta\) is defined by
  \[
(l_1, l_2, \dots, l_n) = (\dfrac{y_1}{x_1}, \dfrac{y_2}{x_2}, \dots, \dfrac{y_n}{x_n}).
\]
  	By assigning a comparison point to each position \(\Theta\in\mathbb T_n\), we obtain a new tree, called the
  	\emph{comparison tree}, which we denote by \(\mathcal L_n\).
  \end{definitionx}

The comparison tree \(\mathcal L_n\) also admits mutation rules. Keeping the notation in
Definition \ref{Comparison point}, consider the mutation in the direction \(i\) at a position \(\Theta\in \mathbb T_n\).

The mutation in \(\mathcal E_n\) is given by
\[
(x_1,x_2,\dots,x_n)
\xrightarrow{\mathcal M_{i;0}}
(x_1,x_2,\dots,x_1+\cdots+\widehat{x_i}+\cdots+x_n,\dots,x_n),
\]
and the mutation in \(\mathcal E_{n,k}\) is given by
\[
(y_1,y_2,\dots,y_n)
\xrightarrow{\mathcal M_{i;k}}
(y_1,y_2,\dots,y_1+\cdots+\widehat{y_i}+\cdots+y_n+k,\dots,y_n).
\]

Set $S_i:=x_1+\cdots+\widehat{x_i}+\cdots+x_n$.
Then the induced mutation in \(\mathcal L_n\) is
\begin{align}
	(l_1,l_2,\dots,l_n)
	\xrightarrow{\Delta_i}&
	\bigl(
	l_1,\dots,l_{i-1},
	\frac{y_1+\cdots+\widehat{y_i}+\cdots+y_n+k}{S_i},
	l_{i+1},\dots,l_n
	\bigr)
	\notag\\
	&=
	\biggl(
	l_1,\dots,l_{i-1},
	\sum_{j\ne i}\frac{x_j}{S_i}l_j+\frac{k}{S_i},
	l_{i+1},\dots,l_n
	\biggr).
	\label{mutation in comparison tree in general}
\end{align}

If we call the interval determined by the numbers
\[
l_1,\dots,l_{i-1},\,l_{i+1},\dots,l_n
\]
the \emph{total interval} and denote it by \(L\), then the mutation \(\Delta_i\) in
\eqref{mutation in comparison tree} may be interpreted as follows: first take a weighted average inside the total interval \(L\), with weights
\[
\frac{x_1}{S_i},\dots,\frac{x_{i-1}}{S_i},\frac{x_{i+1}}{S_i},\dots,\frac{x_n}{S_i},
\]
and then shift the result to the right by \(\frac{k}{S_i}\).

\begin{figure}[htbp]
	\centering
	\begin{tikzpicture}[x=1cm,y=1cm]
		\usetikzlibrary{decorations.pathreplacing,arrows.meta}
		
		\tikzset{
			pt/.style={circle,fill=black,inner sep=1.4pt},
			lbl/.style={font=\small},
			arr/.style={-{Latex[length=2mm]}, thick}
		}
		
		\draw[thick] (0,0) -- (11,0);
		
		\coordinate (Lmin) at (2,0);
		\coordinate (Lmax) at (8,0);
		
		\draw[thick] (2,-0.12) -- (2,0.12);
		\draw[thick] (8,-0.12) -- (8,0.12);
		
		\node[lbl,below=4pt] at (2,0) {$\min\limits_{j\ne i} l_j$};
		\node[lbl,below=4pt] at (8,0) {$\max\limits_{j\ne i} l_j$};
		
		\draw[decorate,decoration={brace,amplitude=5pt}] (2,0.55) -- (8,0.55);
		\node[lbl] at (5,0.95) {total interval \(L\)};
		
		\coordinate (A) at (5.1,0);
		\node[pt,fill=blue] at (A) {};
		\node[lbl,below=8pt,align=center] at (A)
		{$\displaystyle \bar{l}_i=\sum_{j\ne i}\frac{x_j}{S_i}l_j$};
		
		\coordinate (B) at (9.3,0);
		\node[pt,fill=red] at (B) {};
		\node[lbl,below=8pt,align=center] at (B)
		{$\displaystyle l_i'
			=\bar{l}_i+\frac{k}{S_i}$};
		
		\draw[arr,red] (A) to[bend left=18] node[lbl,above]
		{$\displaystyle +\frac{k}{S_i}$} (B);
		
	\end{tikzpicture}
	\caption{Mutation of comparison $n$-tuple at $\Delta_i$}
	\label{fig:comparison-mutation}
\end{figure}
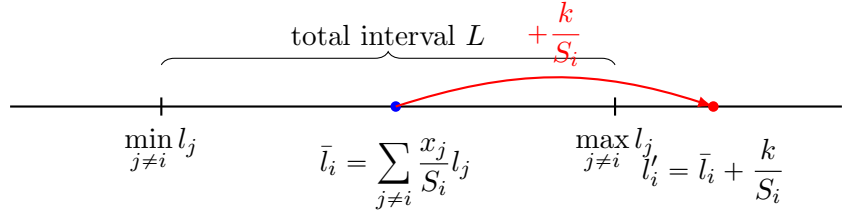

With such phenomenon in mind, together with some property for the Fibonacci sequence, in \cite{CJ25}, the authors proved the boundedness of the comparison points in the case of $n = 3$.
Though the strategy of proof in \cite{CJ25} still works in the case of $n > 3$, in the next subsection, we choose another way of proving the boundedness by introducing a new sequence of points.

\subsection{Boundedness of the comparison points}

In this subsection, we show the boundedness of the comparison points in $\mcL_n$.
Before that, let us fix the following notations:

\begin{notations}\label{Notation of comparison points}
  Under the \Cref{generic assumption}, we take any generic sequence $\mfw \in \mbT_n$, namely $\mfw = [w_1, w_2, \dots]$.
  Following the previous subsection, a sequence of points associated to $\mfw$ in the classical Euclid tree $\mcE_{n}$ is denoted by:
  \begin{align*}
    ({}_0 x_1,{}_0 x_2,\dots,{}_0 x_n) \xrightarrow{\mcM_{w_1;0}} ({}_1 x_1, {}_1 x_2, \dots, {}_1 x_n) \xrightarrow{\mcM_{w_2;0}} ({}_2 x_1, {}_2 x_2, \dots, {}_2 x_n) \xrightarrow{\mcM_{w_3;0}} \cdots
  \end{align*}
  Similarly, a sequence of points associated to $\mfw$ in the $k$-deformed Euclid tree $\mcE_{n,k}$ is denoted by:
  \begin{align*}
    ({}_0 y_1,{}_0 y_2,\dots,{}_0 y_n) \xrightarrow{\mcM_{w_1;k}} ({}_1 y_1, {}_1 y_2, \dots, {}_1 y_n) \xrightarrow {\mcM_{w_2;k}} ({}_2 y_1, {}_2 y_2, \dots, {}_2 y_n) \xrightarrow{\mcM_{w_3;k}} \cdots   
  \end{align*}
  Thus, a sequence of comparison points associated to $\mfw$ in the comparison tree $\mcL_n$ can be denoted by:
  \begin{align*}
    ({}_0 l_1,{}_0 l_2,\dots,{}_0 l_n) \xrightarrow{\Delta_{w_1 ; k}} ({}_1 l_1, {}_1 l_2, \dots, {}_1 l_n) \xrightarrow{\Delta_{w_2 ; k}} ({}_2 l_1, {}_2 l_2, \dots, {}_2 l_n) \xrightarrow{\Delta_{w_3 ; k}} \cdots
  \end{align*}
  From now on, we tend to study the behavior of the sequence of points $\{({}_i l_1, {}_i l_2, \dots, {}_i l_n)\}_{i=0}^{\infty}$.

  Note that as stated in \Cref{Notation for Euclid}, since the mutation rule is given as mutating only one coordinate in each point, we essentially obtain a sequence of numbers $\{{}_i\mathfrak{l}\}_{i=1}^{\infty}$.
  Finally, we denote the total interval associated to the point $({}_i l_1, {}_i l_2, \dots, {}_i l_n)$ by ${}_i L$.
\end{notations}

In the case \(k=0\), the boundedness is immediate.

\begin{proposition}\label{boundness to point without moving}
  Take any generic sequence $\mfw = [w_1, w_2, \dots] \in \mbT_n$, under the condition in \Cref{Notation of comparison points} and $k=0$, 
  we still denote the sequence of comparison points associated to $\mfw$ in the comparison tree $\mcL_n$ by:
  \begin{align*}
    ({}_0 l_1,{}_0 l_2,\dots,{}_0 l_n) \xrightarrow{\Delta_{w_1}} ({}_1 l_1, {}_1 l_2, \dots, {}_1 l_n) \xrightarrow{\Delta_{w_2}} ({}_2 l_1, {}_2 l_2, \dots, {}_2 l_n) \xrightarrow{\Delta_{w_3}} \cdots
  \end{align*}
  Then the sequence of comparison numbers $\{{}_i \mathfrak{l}\}_{i=1}^{\infty}$ is bounded above.
\end{proposition}

We now prove boundedness in the general case by comparison with an auxiliary sequence.
\begin{proposition}[Boundedness of comparison points]
	\label{boundedness of comparison points}
	With the above notation, let \(\mfw=[w_1,w_2,\dots]\in \mbT_n\) be any sequence.
	Then \(\{{}_i\mathfrak l\}_{i=0}^{\infty}\) is bounded.
\end{proposition}

\begin{proof}
	For each \(i\ge 0\) and \(1\le j\le n\), we define an auxiliary sequence $\{({}_i u_1,{}_i u_2,\dots,{}_i u_n)\}_{i=0}^{\infty}$ by
	\[
	{}_i u_j:={}_i l_j+\frac{k}{(n-2)\,{}_i x_j}.
	\]
	We claim that the sequence $\{({}_i u_1,{}_i u_2,\dots,{}_i u_n)\}_{i=0}^{\infty}$ satisfies exactly the same mutation rule as the comparison points in the case \(k=0\).
	
    Set ${}_i S_{w_{i+1}} \coloneqq \sum_{j\ne w_{i+1}} {}_i x_j$.
	Since ${}_{i+1} x_{w_{i+1}}={}_i S_{w_{i+1}}$ by the mutation rule in the classical Euclid tree \(\mcE_n\),
we have
	\begin{align*}
		{}_{i+1}u_{w_{i+1}}
		&=
		{}_{i+1}l_{w_{i+1}}
		+\frac{k}{(n-2)\,{}_{i+1}x_{w_{i+1}}} \\
		&=
		\sum_{j\ne w_{i+1}}
		\frac{{}_i x_j}{{}_i S_{w_{i+1}}}\,{}_i l_j
		+\frac{k}{{}_i S_{w_{i+1}}}
		+\frac{k}{(n-2)\, {}_i S_{w_{i+1}}} \\
		&=
		\sum_{j\ne w_{i+1}}
		\frac{{}_i x_j}{{}_i S_{w_{i+1}}}
		\left(
		{}_i l_j+\frac{k}{(n-2)\,{}_i x_j}
		\right) \\
		&=
		\sum_{j\ne w_{i+1}}
		\frac{{}_i x_j}{{}_i S_{w_{i+1}}}\,{}_i u_j.
	\end{align*}
	For \(j\ne w_{i+1}\), neither \({}_i x_j\) nor \({}_i l_j\) changes under the mutation, and therefore
	\[
	{}_{i+1}u_j={}_i u_j.
	\]
	Thus the sequence \(\{({}_i u_1,\dots,{}_i u_n)\}_{i=0}^{\infty}\) follows exactly the comparison mutation rule in the case \(k=0\).
	
	By Proposition \ref{boundness to point without moving}, the sequence $\{({}_i u_1,{}_i u_2,\dots,{}_i u_n)\}_{i=0}^{\infty}$
	is bounded. Finally, since \({}_i x_j\ge 1\) for all \(i\) and \(j\), we have
	\[
	\bigl|{}_i l_j-{}_i u_j\bigr|
	=
	\frac{k}{(n-2)\,{}_i x_j}
	\le
	\frac{k}{n-2}.
	\]
	Hence the sequence $\{({}_i l_1,{}_i l_2,\dots,{}_i l_n)\}_{i=0}^{\infty}$ is also bounded.
\end{proof}

\begin{remarkx}
	The above argument also works when the parameter \(k\) is allowed to vary at each mutation step, provided that $0\le k_i\le k_c$ for all $i\ge 1$ and some constant \(k_c>0\).
	
	More precisely, let \(\mfw=[w_1,w_2,\dots]\in\mbT_n\) be any mutation sequence and 
	\[
	\{({}_i l_1,{}_i l_2,\dots,{}_i l_n)\}_{i=0}^{\infty}
	\]
	 be the associated sequence of comparison points satisfying
   \[
   ({}_i l_1,\dots,{}_i l_n)
   \xrightarrow{\Delta_{w_{i+1}}}
   ({}_i l_1,\dots,{}_i l_{w_{i+1}-1},
   \sum_{j\ne w_{i+1}}\frac{{}_i x_j}{{}_i S_{w_{i+1}}}\,{}_i l_j+\frac{k_{i+1}}{{}_i S_{w_{i+1}}},
   {}_i l_{w_{i+1}+1},\dots,{}_i l_n),
   \]
   We define an auxiliary sequence
   \[
   \{({}_i u_1,{}_i u_2,\dots,{}_i u_n)\}_{i=0}^{\infty}
   \]
   as follows. Its initial point is given by ${}_0 u_j:={}_0 l_j+\frac{k_c}{(n-2)\,{}_0 x_j}$ for $1\le j\le n$, For \(i\ge 0\), the sequence evolves according to the comparison mutation rule in the case \(k=0\), namely,
   \[
   ({}_i u_1,\dots,{}_i u_n)
   \xrightarrow{\Delta_{w_{i+1}}}
   ({}_i u_1,\dots,{}_i u_{w_{i+1}-1},
   \sum_{j\ne w_{i+1}}\frac{{}_i x_j}{{}_i S_{w_{i+1}}}\,{}_i u_j,
   {}_i u_{w_{i+1}+1},\dots,{}_i u_n).
   \]
   We claim that
   \[
   {}_i u_j\ge {}_i l_j+\frac{k_c}{(n-2)\,{}_i x_j}
   \qquad
   (i\ge 0,\ 1\le j\le n).
   \]
   This is clear for \(i=0\). Assume that it holds for some \(i\). If \(j\ne w_{i+1}\), then nothing changes, and the inequality is immediate. For \(j=w_{i+1}\), we have
   \begin{align*}
   	{}_{i+1}u_{w_{i+1}}
   	&=
   	\sum_{j\ne w_{i+1}}\frac{{}_i x_j}{{}_i S_{w_{i+1}}}\,{}_i u_j \\
   	&\ge
   	\sum_{j\ne w_{i+1}}\frac{{}_i x_j}{{}_i S_{w_{i+1}}}
   	\left(
   	{}_i l_j+\frac{k_c}{(n-2)\,{}_i x_j}
   	\right) \\
   	&=
   	\sum_{j\ne w_{i+1}}\frac{{}_i x_j}{{}_i S_{w_{i+1}}}\,{}_i l_j
   	+\frac{(n-1)k_c}{(n-2) {}_i S_{w_{i+1}}} \\
   	&\ge
   	\sum_{j\ne w_{i+1}}\frac{{}_i x_j}{{}_i S_{w_{i+1}}}\,{}_i l_j
   	+\frac{k_{i+1}}{{}_i S_{w_{i+1}}}
   	+\frac{k_c}{(n-2)\, {}_i S_{w_{i+1}}} \\
   	&=
   	{}_{i+1}l_{w_{i+1}}
   	+\frac{k_c}{(n-2)\,{}_{i+1}x_{w_{i+1}}},
   \end{align*}
   since \(k_{i+1}\le k_c\) and \({}_{i+1}x_{w_{i+1}} = {}_i S_{w_{i+1}}\). 
	 Hence the claim follows by Proposition \ref{boundedness of comparison points}.
\end{remarkx}

Once the boundedness of the sequence of comparison numbers has been established, together with the convergence of the total interval introduced below, we can derive the asymptotic relation between the \(k\)-deformed Euclid tree and the classical Euclid tree.

In \cite{CJ25}, we proved the convergence directly in the case \(n=3\). However, when \(n\ge 4\), the same method no longer works. To overcome this difficulty, we give a new (combinatorial) approach to prove the convergence of the total interval.

\subsection{Convergence of the total interval}
In this subsection, we show that the length of the total interval associated with the sequence of comparison points converges to \(0\).

We first consider the simpler case in which the comparison tree is defined by two classical \(n\)-branched Euclid trees starting from two different initial points. 
In this setting, using the same notation as in \Cref{Notation of comparison points} with \(k=0\), the mutation of the comparison point in direction \(i\) at a given position simplifies to the following:
\[
	(l_1,l_2,\dots,l_n)
	\xrightarrow{\Delta_i}
	\biggl(
	l_1,\dots,l_{i-1},
	\sum_{j\ne i}\frac{x_j}{S_i}l_j,
	l_{i+1},\dots,l_n
	\biggr).
	\label{mutation in comparison tree}
\]
where $S_i:=x_1+\cdots+\widehat{x_i}+\cdots+x_n$.
The following proposition shows that the sequence of total intervals associated to the sequence of comparison points converges to a point under the above mutation rule.
The proof of this proposition is purely combinatorial, and it is the key to show the convergence of the total interval in the general case when $k>0$.

\begin{proposition}\label{converge to point without moving}
	Let $k=0$ and \(\mfw=[w_1,w_2,\dots]\in\mbT_n\) be a generic sequence. With the notation of
	(\ref{Notation of comparison points}), let
	\[
	({}_0 l_1,\dots,{}_0 l_n)\xrightarrow{\Delta_{w_1}}
	({}_1 l_1,\dots,{}_1 l_n)\xrightarrow{\Delta_{w_2}}
	({}_2 l_1,\dots,{}_2 l_n)\xrightarrow{\Delta_{w_3}}\cdots
	\]
	be the associated sequence of comparison points in the comparison tree \(\mcL_n\).
	For each \(i\ge 0\), let \({}_i L\) denote the total interval associated to
	\(({}_i l_1,\dots,{}_i l_n)\).
	
	Then the sequence of total intervals \(\{{}_iL\}_{i=0}^{\infty}\) converges to a point.
\end{proposition}

\begin{proof}
Write ${}_i L=[{}_i l_{\min},\,{}_i l_{\max}]$ and $|{}_i L|:={}_i l_{\max}-{}_i l_{\min}$, where ${}_i l_{\max} = \max \{{}_i l_1, {}_i l_2, \dots, {}_i l_n\}$ and ${}_i l_{\min} = \min \{{}_i l_1, {}_i l_2, \dots, {}_i l_n\}$.
Since in the case \(k=0\) every new coordinate is obtained as a weighted average of the
remaining coordinates, we have
\[
{}_{i+1} L\subseteq {}_i L
\qquad\text{for all }i\ge 0.
\]
\(\{{}_i L\}_{i\ge 0}\) converges to some interval \(L = [l_\alpha, l_\beta]\).
We claim that \(|L|=0\). Suppose, for contradiction, that \(|L|>0\).
	
	Fix \(0<\varepsilon<|L|\). 
  Since \({}_i L\to L\), there exists \(N\in\mathbb N\) such that for all
	\(t\ge N\),
	\[
	0<l_\alpha-{}_t l_{\min}<\varepsilon
	\qquad\text{and}\qquad
	0<{}_t l_{\max}-l_\beta<\varepsilon.
	\]
	For each \(t\ge N\), let \(A_{t;<}:= \{  {}_t l_{j} \mid {}_t l_{j} <l_{\alpha} \}\) and \(A_{t;>} := \{  {}_t l_{j} \mid {}_t l_{j} >l_{\beta} \} \).
	
	We now establish the key observation.
	
	\medskip

  \noindent
  \textbf{Claim.}
  Suppose that \(t_0\ge N\), and that the mutation \(\Delta_{w_{t_0+1}}\) produces a new coordinate lying to the left of \(l_\alpha\), that is,
  \[
    {}_{t_0+1}l_{w_{t_0+1}} \in A_{t_0+1;<}.
  \]
  Let \(t_1>t_0\) be the first step after \(t_0\) at which a newly produced coordinate lies to the right of \(l_\beta\), that is,
  \[
    {}_{t_1+1}l_{w_{t_1+1}} \in A_{t_1+1;>}.
  \]
  Then the set of coordinates lying to the left of \(l_\alpha\) strictly decreases:
  \[
    A_{t_1;<}\subsetneq A_{t_0;<}.
  \]
	
	\medskip
	\noindent
	\emph{Proof of the claim.} 
	For \(j=0,1\), let
	\[
	\{i_1,\dots,i_{n-1}\}
	=
	\{1,\dots,n\}\setminus\{w_{t_j+1}\},
	\]
	and reorder them so that
	\[
	{}_{t_j}l_{i_1},\dots,{}_{t_j}l_{i_{h_j}}<l_\alpha\leq {}_{t_j}l_{i_{h_j+1}},\dots,{}_{t_j}l_{i_{m_j-1}}\leq l_\beta<{}_{t_j}l_{i_{m_j}},\dots,{}_{t_j}l_{i_{n-1}},
	\]
	for suitable integers \(h_j\) and \(m_j\).
	Since \({}_{t_0+1} l_{w_{t_0+1}} \in A_{t_0+1;<}\), we have 
	\[
	{}_{t_0+1}l_{w_{t_0+1}}  = \sum_{i\neq w_{t_0+1}}\frac{{}_{t_0}x_i}{{}_{t_0+1}x_{w_{t_0+1}}}{}_{t_0}l_i < l_\alpha.
	\]
	Equivalently, 
	\begin{align}
		\sum_{r=h_0+1}^{n-1} {}_{t_0}x_{i_r}\bigl({}_{t_0}l_{i_r}-l_\alpha\bigr)
	<
	\sum_{r=1}^{h_0} {}_{t_0}x_{i_r}\bigl(l_\alpha-{}_{t_0}l_{i_r}\bigr)
	\end{align}
	since ${}_{t_0+1}x_{w_{t_0+1}} = \sum_{i\neq w_{t_0+1}} {}_{t_0}x_i$.
  From \(l_\alpha-{}_{t_0}l_{i_r}<\varepsilon < |L|\) for \(1\le r\le h_0\), it follows that
	\[
	\sum_{r=1}^{h_0} {}_{t_0}x_{i_r}\bigl(l_\alpha-{}_{t_0}l_{i_r}\bigr)
	<
	\varepsilon \sum_{r=1}^{h_0} {}_{t_0}x_{i_r}
	< {}_{t_0+1}x_{w_{t_0+1}} |L|.
	\]
	Discarding the nonnegative middle contribution, we obtain
	\begin{align}\label{in;t_0}
	\sum_{r=m_0}^{n-1} {}_{t_0}x_{i_r}\bigl({}_{t_0}l_{i_r}-l_\beta\bigr)
	\leq \sum_{r=m_0}^{n-1} {}_{t_0}x_{i_r}\bigl({}_{t_0}l_{i_r}-l_\alpha\bigr) <\sum_{r=1}^{h_0} {}_{t_0}x_{i_r}\bigl(l_\alpha-{}_{t_0}l_{i_r}\bigr)
	<  {}_{t_0+1}x_{w_{t_0+1}}|L| 
	\end{align}
	since 	\({}_{t_0+1}l_{w_{t_0+1}} <l_\alpha\).
	
	On the other hand, we have \({}_{t_1+1}l_{w_{t_1+1}} \in A_{t_1+1;>}\). Similarly, it implies that 
	\begin{align}\label{in;t_1}
		\sum_{r=m_1}^{n-1} {}_{t_1}x_{i_r}\bigl({}_{t_1}l_{i_r}-l_\beta\bigr)
	>
	\sum_{r=1}^{m_1-1} {}_{t_1}x_{i_r}\bigl(l_\beta-{}_{t_1}l_{i_r}\bigr)
	\geq
	\sum_{r=1}^{h_1} {}_{t_1}x_{i_r}\bigl(l_\beta-{}_{t_1}l_{i_r}\bigr)
	.
	\end{align}
	From the hypothesis in the claim, we have ${}_{t}l_{{w_t}} < l_\beta$ for every $t_0< t \leq t_1$, then
	\[
	\sum_{r=m_0}^{n-1} {}_{t_0}x_{i_r}\bigl({}_{t_0}l_{i_r}-l_\beta\bigr)
	\geq\sum_{r=m_1}^{n-1} {}_{t_1}x_{i_r}\bigl({}_{t_1}l_{i_r}-l_\beta\bigr).
	\]
	by (\ref{in;t_0}) and (\ref{in;t_1}).

	Therefore, we have 
	\begin{align}\label{in:t_0t_1}
			{}_{t_0+1} x_{w_{t_0+1}}|L| > \sum_{r=1}^{h_1} {}_{t_1} x_{i_r} \bigl( l_\beta-{}_{t_1} l_{i_r} \bigr)
	\end{align}
	
	By the mutation rule of classical Euclid tree, ${}_{i} x_{w_{i}}\leq{}_{i+1} x_{w_{i+1}}$ for every $i\geq 1$. 
	If ${}_{t}l_{{w_t}}\in A_{t_1;<}$ for some $t_0\leq t \leq t_1$, then (\ref{in:t_0t_1}) does not hold since $l_\beta-{}_{t_1}l_{i_r}>L$ for every $1\leq r \leq h_1$, so $A_{t_1; <} \subset A_{t_0;<}$.
	
	If $A_{t_1; <} = A_{t_0;<}$, then 
	\[
	\sum_{r=1}^{h_0} {}_{t_0}x_{i_r}\bigl(l_\alpha-{}_{t_0}l_{i_r}\bigr) > \sum_{r=1}^{h_1} {}_{t_1}x_{i_r}\bigl(l_\beta-{}_{t_1}l_{i_r}\bigr)
	= \sum_{r=1}^{h_0} {}_{t_0}x_{i_r}\bigl(l_\beta-{}_{t_0}l_{i_r}\bigr)
	\]
	by (\ref{in;t_0}) and (\ref{in;t_1}). This is impossible. Hence, $A_{t_1; <} \subsetneq A_{t_0;<}$, that is, the claim holds. 

	We now derive the contradiction from $|L|>0$. Suppose that \(t_0\ge N\), and that the mutation $\Delta_{w_{t_0+1}}$ produces the new coordinate
	\({}_{t_0+1}l_{w_{t_0+1}} \in A_{t_0+1;<}\).  There are two possibilities.
	
	\noindent
	\emph{Case 1.} For all \(t>t_0\), ${}_t l_{w_t}\leq l_\beta$.
Since \(\mfw\) is generic, every direction is mutated infinitely many times. Hence there exists an infinite sequence $t_1<t_2<t_3<\cdots $
such that
\[
A_{t_{i+1};>} \subsetneq A_{t_i;>}
\qquad\text{for all } i\ge 1.
\]
This is impossible, because each \(A_{t_i;>}\) is nonempty.
	
\noindent
\emph{Case 2.} There exists \(t_1>t_0\) such that ${}_{t_1+1} l_{w_{t_1+1}} > l_\beta$.
Choose \(t_1\) to be the first such number. Then, by the claim proved above,
$A_{t_1;<}\subsetneq A_{t_0;<}$. 

If for all \(t>t_1\), ${}_t l_{w_t}\geq l_\alpha$, then we are in the symmetric version of Case~1, which is impossible. Hence there exists \(t_2>t_1\) such that ${}_{t_2+1}l_{w_{t_2+1}}<l_\alpha$.
Taking \(t_2\) to be the first such number, the symmetric claim yields $A_{t_2;>}\subsetneq A_{t_1;>}$. In particular, we have $A_{t_2;<}= A_{t_1;<}$.

Continuing in this way, we obtain an increasing sequence
\[
t_0<t_1<t_2<t_3<\cdots
\]
such that the process alternates between producing a new coordinate which is strictly greater than $l_\beta$ or strictly less than $l_\alpha$, and at each such step one of the two sets \(A_{t;<}\) or \(A_{t;>}\) strictly decreases. More precisely,
\begin{align*}
A_{t_{2r+1};<}\subsetneq A_{t_{2r};<}=A_{t_{2r-1};<},
\ 
A_{t_{2r+2};>}\subsetneq A_{t_{2r+1};>} = A_{t_{2r};>}, \ \dots
\end{align*}
whenever these indices are defined. This is a contradiction.
	
This contradiction shows that our assumption \(|L|>0\) is false. 
Therefore $|L|=0$. Equivalently, the sequence of total intervals \(\{{}_i L\}_{i=0}^{\infty}\) converges to a point.
\end{proof}

By the above proposition, we can also prove the following corollary for the case \(k\neq 0\), which states that the total interval associated with the sequence of comparison points converges to a point.

\begin{corollary}[Convergence of total interval]\label{convergence of total interval}
  Take any generic sequence $\mfw = [w_1, w_2, \dots] \in \mbT_n$, under the condition in \Cref{Notation of comparison points}, 
  we still denote the sequence of comparison points associated to $\mfw$ in the comparison tree $\mcL_n$ by:
  \begin{align*}
    ({}_0 l_1,{}_0 l_2,\dots,{}_0 l_n) \xrightarrow{\Delta_{w_1}} ({}_1 l_1, {}_1 l_2, \dots, {}_1 l_n) \xrightarrow{\Delta_{w_2}} ({}_2 l_1, {}_2 l_2, \dots, {}_2 l_n) \xrightarrow{\Delta_{w_3}} \cdots
  \end{align*}
  and the total interval associated to the point $({}_i l_1, {}_i l_2, \dots, {}_i l_n)$ by ${}_i L$.

  Then the sequence of total intervals $\{{}_i L\}_{i=1}^{\infty}$ converges to a point.
\end{corollary}

\begin{proof}
  For any $t \in \mbN$, we obtain a comparison point $({}_t l_1, {}_t l_2, \dots, {}_t l_n)$, and then consider the mutations after $t$.
  
  With respect to the sequence $[w_{t+1}, w_{t+2}, \dots]$,  recall the sequence  $\{({}_i u_1, {}_i u_2, \dots, {}_i u_n)\}_{i=t}^{\infty}$ defined as the follows:
  \begin{align*}
    {}_i u_j \coloneqq {}_i l_j + \dfrac{k}{(n-2)\times {}_i x_j}.
  \end{align*}
  We have the following inequality:
  \begin{align*}
    | {}_i u_j - {}_i l_j | = \dfrac{k}{(n-2)\times {}_i x_j} < \dfrac{k}{(n-2)\times {}_t x_j}
  \end{align*}
  for every $i > t$.
  
  Fix \(\varepsilon>0\). 
	Choose \(t\) such that
  \[
  \frac{2k}{(n-2)\,{}_t x_{\min}}<\frac{\varepsilon}{2}.
  \]
 From the proof of \Cref{boundedness of comparison points}, we know that 
\begin{align*}
	\lim_{i \to \infty} ({}_i u_{\max} - {}_i u_{\min}) = 0.
\end{align*}
  Hence, there exists \(N\ge t\) such that for all \(i\ge N\),
  \[
  {}_i u_{\max}-{}_i u_{\min}<\frac{\varepsilon}{2}.
  \]
  For each \(i\ge N\), choose indices \(s_i,t_i\) such that ${}_i l_{s_i} \coloneqq {}_i l_{\max}$ and ${}_i l_{t_i} \coloneqq {}_i l_{\min}$.
  Then, we obtain that 
  \begin{align*}
  	|{}_i L|
  	&=
  	{}_i l_{\max}-{}_i l_{\min}\\
  	&\le
  	|{}_i l_{s_i}-{}_i u_{s_i}|
  	+
  	|{}_i u_{s_i}-{}_i u_{t_i}|
  	+
  	|{}_i u_{t_i}-{}_i l_{t_i}|\\
  	&\le
  	\frac{k}{(n-2)\,{}_t x_{\min}}
  	+
  	({}_i u_{\max}-{}_i u_{\min})
  	+
  	\frac{k}{(n-2)\,{}_t x_{\min}}\\
  	&<
  	\frac{\varepsilon}{2}+\frac{\varepsilon}{2}
  	=
  	\varepsilon.
  \end{align*}
  Therefore, the sequence of total intervals $\{{}_i L\}_{i=1}^{\infty}$ converges to a point.
\end{proof}

\begin{remarkx}
  As before, the above proof also works when $k \geq 0$ varies at each mutation, provided that $k$ is bounded below by some constant $k_a \geq 0$.
  In other words, for any generic mutation chain $\mfw \in \mbT_n$, we associate to it a sequence of numbers $\{k_i\}_{i=1}^{\infty}$ bounded below by $k_a$, and bounded above by $k_c$.

  Indeed, after $t$-th step, denote by $\{({}_i l_1, {}_i l_2, \dots, {}_i l_n)\}_{i = t}^{\infty}$ the sequence of points associated with $[w_{t+1}, w_{t+2}, \dots]$ under the mutation rule
  \begin{align}
    ({}_i l_1, {}_i l_2, \dots, {}_i l_n)  \xrightarrow{\Delta_{w_i}}  ({}_i l_1, \dots, {}_i l_{w_i - 1}, \sum_{j \neq w_i} \dfrac{{}_i x_j}{{}_i S_{w_i}} \times {}_i l_j + \dfrac{k_i}{{}_i S_{w_i}}, {}_i l_{w_i + 1}, \dots, {}_i l_n).
  \end{align}

  Let $\{({}_i l_1^{\prime}, {}_i l_2^{\prime}, \dots, {}_i l_n^{\prime})\}_{i = t}^{\infty}$ be the sequence of points with constant right-shift parameter $k=k_a$:
  \begin{align}
    ({}_i l_1^{\prime}, {}_i l_2^{\prime}, \dots, {}_i l_n^{\prime})  \xrightarrow{\Delta_{w_i}}  ({}_i l_1^{\prime}, \dots, {}_i l_{w_i - 1}^{\prime}, \sum_{j \neq w_i} \dfrac{{}_i x_j}{{}_i S_{w_i}} \times {}_i l_j^{\prime} + \dfrac{k_a}{{}_i S_{w_i}}, {}_i l_{w_i + 1}^{\prime}, \dots, {}_i l_n^{\prime}).
  \end{align}

	Let $\{({}_i l_1^{\prime \prime}, {}_i l_2^{\prime \prime}, \dots, {}_i l_n^{\prime \prime})\}_{i = t}^{\infty}$ be the sequence of points with constant right-shift parameter $k=k_c$:
  \begin{align}
    ({}_i l_1^{\prime \prime}, {}_i l_2^{\prime \prime}, \dots, {}_i l_n^{\prime \prime})  \xrightarrow{\Delta_{w_i}}  ({}_i l_1^{\prime \prime}, \dots, {}_i l_{w_i - 1}^{\prime \prime}, \sum_{j \neq w_i} \dfrac{{}_i x_j}{{}_i S_{w_i}} \times {}_i l_j^{\prime \prime} + \dfrac{k_c}{{}_i S_{w_i}}, {}_i l_{w_i + 1}^{\prime \prime}, \dots, {}_i l_n^{\prime \prime}).
  \end{align}

  Note that we have $({}_t l_1, {}_t l_2, \dots, {}_t l_n) = ({}_t l_1^{\prime}, {}_t l_2^{\prime}, \dots, {}_t l_n^{\prime}) = ({}_t l_1^{\prime \prime}, {}_t l_2^{\prime \prime}, \dots, {}_t l_n^{\prime \prime})$.

	Define the sequence of points $\{({}_i u_1^{\prime}, {}_i u_2^{\prime}, \dots, {}_i u_n^{\prime})\}_{i=t}^{\infty}$ by
  \begin{align}
    {}_i u_j^{\prime } \coloneqq {}_i l_j^{\prime} + \dfrac{k_a}{(n-2)\times {}_i x_j}.
  \end{align}

  Again, define the sequence of points $\{({}_i u_1^{\prime \prime}, {}_i u_2^{\prime \prime}, \dots, {}_i u_n^{\prime \prime})\}_{i=t}^{\infty}$ by
  \begin{align}
    {}_i u_j^{\prime \prime} \coloneqq {}_i l_j^{\prime \prime} + \dfrac{k_c}{(n-2)\times {}_i x_j}.
  \end{align}

  Since ${}_i l_j^{\prime} \leq {}_i l_j \leq {}_i l_j^{\prime \prime}$ for any $i > t$ and $1 \leq j \leq n$, we have
  \begin{align}
    | {}_i u_j^{\prime \prime} - {}_i l_j | \leq | {}_i u_j^{\prime \prime} - {}_i l_j^{\prime} | \leq | {}_i u_j^{\prime \prime} - {}_i u_j^{\prime} | + | {}_i u_j^{\prime} - {}_i l_j^{\prime} | \notag \\
		= \dfrac{k_c - k_a}{(n-2)\times {}_t x_j} + \dfrac{k_a}{(n-2)\times {}_i x_j} < \dfrac{k_c}{(n-2)\times {}_t x_j}
  \end{align}
  for every $i > t$.

  Thus the strategy of the proof of the preceding corollary still applies.
  Hence the sequence of lengths of total intervals $\{ |{}_i L| \}_{i=1}^{\infty}$ converges to $0$.
\end{remarkx}

\subsection{Asymptotic behavior of $k$-deformed Euclid trees}

For later use, we need the following definition to state the theorem in full generality.

\begin{definitionx}
  Let $\mfw \in \mbT_n$ be a generic sequence, namely $\mfw = [w_1, w_2, \dots]$.
  Recall that a sequence of points associated with $\mfw$ in the classical Euclid tree $\mcE_{n}$ is denoted by:
  \begin{align*}
    ({}_0 x_1,{}_0 x_2,\dots,{}_0 x_n) \xrightarrow{\mcM_{w_1;0}} ({}_1 x_1, {}_1 x_2, \dots, {}_1 x_n) \xrightarrow{\mcM_{w_2;0}} ({}_2 x_1, {}_2 x_2, \dots, {}_2 x_n) \xrightarrow{\mcM_{w_3;0}} \cdots
  \end{align*}

  Associate to $\mfw$ a sequence of numbers $\{k_i\}_{i=1}^{\infty}$ bounded by some constant $k_c > 0$.
  We then define a sequence of points associated with $\mfw$ by
  \begin{align*}
    ({}_0 y_1,{}_0 y_2,\dots,{}_0 y_n) \xrightarrow{\mcM_{w_1;k_1}} ({}_1 y_1, {}_1 y_2, \dots, {}_1 y_n) \xrightarrow {\mcM_{w_2;k_2}} ({}_2 y_1, {}_2 y_2, \dots, {}_2 y_n) \xrightarrow{\mcM_{w_3;k_3}} \cdots   
  \end{align*}
    where the mutation rule $\mcM_{w_i;k_i}$ is given by
  \begin{align}
    ({}_i y_1, {}_i y_2, \dots, {}_i y_n)  \xrightarrow{\mcM_{w_{i+1};k_{i}}}  ({}_i y_1, \dots, {}_i y_{w_{i+1} - 1}, \sum_{j \neq w_{i+1}} {}_i y_j + k_{i}, {}_i y_{w_{i+1} + 1}, \dots, {}_i y_n).
  \end{align}

  We do not require the initial point $({}_0 y_1,{}_0 y_2,\dots,{}_0 y_n)$ to be the same as the initial point $({}_0 x_1,{}_0 x_2,\dots,{}_0 x_n)$.

   We call such a sequence of points $\{({}_i y_1, {}_i y_2, \dots, {}_i y_n)\}_{i = 0}^{\infty}$ the \emph{$\{k_i\}_{i=1}^{\infty}$-deformed Euclid sequence along $\mfw$}.
\end{definitionx}

Combining \Cref{boundedness of comparison points} with \Cref{convergence of total interval}, we conclude this section with the following theorem.

\begin{theorem}\label{q times of deformed Euclid}
  Let $\mfw = [w_1, w_2, \dots, w_t, \dots] \in \mbT_n$ be a generic mutation sequence, and associate it with a sequence of positive real numbers $\{k_i\}_{i=1}^{\infty}$ bounded above by some constant $k_c > 0$.
  Then there exists a real number $q$ such that the $\{k_i\}_{i=1}^{\infty}$-deformed Euclid sequence along $\mfw$ is asymptotic to $q$ times the corresponding classical Euclid sequence.
\end{theorem}

\section{Logarithmic asymptotic behavior of generalized Markov-Hurwitz equations}

In this section, following the steps in our previous work \cite[Section 6]{CJ25},
we show that after taking logarithms, the generalized Markov-Hurwitz tree converges to the classical $n$-branched Euclid tree up to a scalar multiplication.

\subsection{Ratio number sequence}
\label{ratio number sequence}
In this subsection, we introduce the ratio number sequence, which characterizes the asymptotic behavior of generalized Markov-Hurwitz points.

We begin with an example in which, after $t$-th step for some $t \in \mbN$, the mutation sequence is given by $[\dots,w_t,1,2,\dots]$.

At $t$-th step, let $({}_t x_1, {}_t x_2, \dots, {}_t x_n)$ be a solution to the generalized Markov-Hurwitz equations (\ref{eq: generalized Markov Hurwitz equations}).
At $(t+1)$-st step, following the mutation chain $[\dots,w_t,1,2,\dots]$, we mutate the first coordinate and then obtain the new point $({}_{t+1} x_1, {}_{t+1} x_2, \dots, {}_{t+1} x_n)$, where
\begin{align}
    {}_{t+1} x_1 = \dfrac{({}_t x_2^2 + \cdots + {}_t x_n^2) + \lambda_1 \cdot {}_t x_2 \cdots {}_t x_n}{{}_t x_1}, \quad {}_{t+1} x_2 = {}_t x_2, \quad \dots, \quad {}_{t+1} x_n = {}_t x_n.
\end{align}
We define the ratio number $k_{t+1}$ by
\begin{align}
    k_{t+1} = \dfrac{{}_{t+1} x_1}{\prod_{i \neq 1} {}_t x_i}.
\end{align}
That is to say, the ratio number is the newly mutated coordinate divided by the product of the remaining coordinates at the previous step.

Next, at $(t+2)$-nd step, we mutate the second coordinate of $({}_{t+1} x_1, {}_{t+1} x_2, \dots, {}_{t+1} x_n)$ and obtain the point $({}_{t+2} x_1, {}_{t+2} x_2, \dots, {}_{t+2} x_n)$, where ${}_{t+2} x_i = {}_{t+1} x_i$ for $i \neq 2$, and
\begin{align}
  {}_{t+2} x_2 = \dfrac{({}_{t+1} x_1^2 + {}_{t+1} x_3^2 + \cdots + {}_{t+1} x_n^2) + \lambda_2 \cdot {}_{t+1} x_1 \cdot {}_{t+1} x_{3} \cdot {}_{t+1} x_{4} \cdots {}_{t+1} x_n}{{}_{t+1} x_2}.
\end{align}

Similarly, the new ratio number $k_{t+2}$ is then defined by 
{\small
\begin{align}
    k_{t+2} &= \dfrac{{}_{t+2} x_2}{\prod_{i \neq 2} {}_{t+1} x_i} = \dfrac{({}_{t+1} x_1^2 + {}_{t+1} x_3^2 + \cdots + {}_{t+1} x_n^2) + \lambda_2 \cdot {}_{t+1} x_1 \cdot {}_{t+1} x_{3} \cdot {}_{t+1} x_{4} \cdots {}_{t+1} x_n}{{}_{t+1} x_1 \cdot \prod_{i \neq 1} {}_{t+1} x_i } \notag \\
    &= \dfrac{(k_{t+1} \cdot \prod_{i \neq 1} {}_t x_i)^2 + {}_{t+1} x_3^2 + \cdots + {}_{t+1} x_n^2 + \lambda_2 \cdot (k_{t+1} \cdot \prod_{i \neq 1} {}_t x_i) \cdot {}_{t+1} x_{3} \cdot {}_{t+1} x_{4} \cdots {}_{t+1} x_n}    {(k_{t+1} \cdot \prod_{i \neq 1} {}_t x_i) \cdot \prod_{i \neq 1} {}_t x_i} \notag \\
    &= k_{t+1} + \dfrac{{}_{t+1} x_3^2 + \cdots + {}_{t+1} x_n^2}   {\prod_{i \neq 1} {}_{t+1} x_i^2} \cdot \dfrac{1}{k_{t+1}} + \lambda_2 \cdot \dfrac{1}{{}_{t+1} x_2}.
\end{align}}

Thus, to any infinite generic mutation sequence $\mfw$ with $|\mfw|=+\infty$, we can associate a number sequence $\{k_j\}_{j=1}^{+\infty}$.
We call this sequence the \emph{ratio number sequence}.
Starting from the initial solution $(1,1,\dots,1)$, one can prove by induction that $k_i> 0$ for every $i\in \mbN$.
Moreover, the ratio sequence becomes stable when the number of mutations is sufficiently large. That is to say, $k_i \approx k_{i+1}$ for $i\gg 0$.

The following lemma follows from a direct calculation and induction.

\begin{lemma}\label{lem: ratio increase}
	The ratio number sequence $\{k_j\}_{j=1}^{+\infty}$ is a strictly increasing sequence.
\end{lemma}

Using the ratio number sequence, we obtain the following observation on the transition behavior of generalized Markov-Hurwitz points.

\begin{observation}\label{product form}
  Take any mutation sequence $\mfw = [w_1, w_2, \dots] \in \mbT_n$ (not necessarily generic), and associate it with the sequence of generalized Markov-Hurwitz points $\{({}_i x_1, {}_i x_2, \dots, {}_i x_n)\}_{i=0}^{\infty}$,
  starting at the initial solution $({}_0 x_1, {}_0 x_2, \dots, {}_0 x_n) = (1,1,\dots,1)$.
  Then, there exists a ratio number sequence $\{k_{i}\}_{i=1}^{\infty}$, such that the mutation $\mu_{w_{i+1}}$ of $({}_i x_1, {}_i x_2, \dots, {}_i x_n)$ is of the form
  \begin{align}
    \mu_{w_{i+1}} ({}_i x_1, {}_i x_2, \dots, {}_i x_n) = ({}_i x_1, {}_i x_2, \dots, {}_i x_{w_{i+1}-1}, k_{i+1} \cdot \prod_{j \neq w_{i+1}} {}_i x_j, {}_i x_{w_{i+1}+1}, \dots, {}_i x_n).
  \end{align}
\end{observation}

\subsection{Logarithmic generalized Markov-Hurwitz trees}

We have seen in the previous subsection that, for the mutation sequences to be sufficiently long, mutations of generalized Markov-Hurwitz points are asymptotically multiplicative.

It is therefore natural to investigate the behavior of mutations after taking logarithms of the generalized Markov-Hurwitz tree, replacing each generalized Markov-Hurwitz point $(x_1, x_2, \dots, x_n)$ by $(\log(x_1), \log(x_2), \dots, \log(x_n))$.
We call the resulting tree the \emph{logarithmic generalized Markov-Hurwitz tree}.

We will compare this logarithmic generalized Markov-Hurwitz tree with the $k$-deformed $n$-branched Euclid tree.

For brevity, we denote by $\widebar{x} \coloneqq \log(x)$ for any positive real number $x$. 
Therefore, we have 
\begin{align}
 	(\widebar{x_1},\widebar{x_2},\dots, \widebar{x_n})\coloneqq (\log(x_1),\log(x_2),\dots, \log(x_n)).
 \end{align}
Take any mutation chain in the generalized Markov-Hurwitz tree.
For example, suppose that the mutation chain is
{\small
\begin{align}
  ({}_0 x_1, {}_0 x_2, \dots, {}_0 x_n) \xrightarrow{\mu_{i_1}} ({}_1 x_1, {}_1 x_2, \dots, {}_1 x_n) \xrightarrow{\mu_{i_2}} ({}_2 x_1, {}_2 x_2, \dots, {}_2 x_n) \xrightarrow{\mu_{i_3}} ({}_3 x_1, {}_3 x_2, \dots, {}_3 x_n) \cdots 
\end{align}} 
Then the corresponding mutation chain in the logarithmic generalized Markov-Hurwitz tree can be written as
{\small
\begin{align}
  (\widebar{{}_0 x_1}, \widebar{{}_0 x_2}, \dots, \widebar{{}_0 x_n}) \xrightarrow{\widebar{\mu_{i_1}}} (\widebar{{}_1 x_1},\widebar{{}_1 x_2}, \dots, \widebar{{}_1 x_n}) \xrightarrow{\widebar{\mu_{i_2}}} (\widebar{{}_2 x_1},\widebar{{}_2 x_2}, \dots, \widebar{{}_2 x_n}) \xrightarrow{\widebar{\mu_{i_3}}} (\widebar{{}_3 x_1},\widebar{{}_3 x_2}, \dots, \widebar{{}_3 x_n}) \cdots 
\end{align}}
Here, we replace $\mu_i$ by $\widebar{\mu_i}$ to denote the mutation in the logarithmic generalized Markov-Hurwitz tree.

Indeed, given any sequence of generalized Markov-Hurwitz points $\{({}_i x_1,{}_i x_2, \dots, {}_i x_n)\}_{i=0}^{\infty}$ associated with a mutation sequence $\mfw$, the definition of the ratio number sequence implies that, if at time $t$ we mutate in direction $j$, then
{\small
\begin{align}
  ({}_{t+1} x_{1}, {}_{t+1} x_{2}, \dots, {}_{t+1} x_{n}) = \mu_j({}_t x_1, {}_t x_2, \dots, {}_t x_n) = ({}_t x_1, \dots, {}_t x_{j-1}, k_{t+1} \cdot \prod_{i \neq j} {}_t x_i, {}_t x_{j+1}, \dots, {}_t x_n),
\end{align}} 
which implies that 
\begin{align}
  (\widebar{{}_{t+1} x_{1}}, \dots, \widebar{{}_{t+1} x_{n}})= (\log({}_{t+1} x_{1}), \dots, \log({}_{t+1} x_{n})) \notag \\
  = (\widebar{{}_t x_1}, \dots, \widebar{{}_t x_{j-1}}, \sum_{i \neq j} \widebar{{}_t x_i} + \widebar{k_{t+1}}, \widebar{{}_t x_{j+1}}, \dots, \widebar{{}_t x_n}).
\end{align} 
Hence, at $t$-th step for any $t \in \mbN$, the mutation $\widebar{\mu_j}$ can be written as
\begin{align}
  \widebar{\mu_j}(\widebar{{}_{t+1} x_{1}}, \dots, \widebar{{}_{t+1} x_{n}}) = (\widebar{{}_t x_1}, \dots, \widebar{{}_t x_{j-1}}, \sum_{i \neq j} \widebar{{}_t x_i} + \widebar{k_{t+1}}, \widebar{{}_t x_{j+1}}, \dots, \widebar{{}_t x_n}),
\end{align}
for any $1 \leq j \leq n$.

Thus the logarithmic mutation rule may be viewed as a deformed Euclid mutation, whose deformation parameter at the $(t+1)$-st step is \(\widebar{k_{t+1}}\).
To apply the asymptotic result for deformed Euclid trees, it remains to understand the limiting behavior of the ratio number sequence \(\{k_i\}_{i=1}^{\infty}\).
The next theorem gives this convergence, and identifies the limit in the generic case.

\begin{theorem}\label{ratio number convergence}
  For any mutation sequence $\mfw = [w_1, w_2, \dots] \in \mbT_n$, the associated ratio number sequence $\{k_i\}_{i=1}^{\infty}$ converges to a real number.
	Moreover, if $\mfw$ is generic, then the associated ratio number sequence $\{k_i\}_{i=1}^{\infty}$ converges to $k_{\lambda} = n + \displaystyle \sum_{i=1}^{n} \lambda_i$.
\end{theorem}

\begin{proof}
  First, by \Cref{lem: ratio increase}, for any mutation sequence $\mfw$, the associated ratio number sequence $\{k_i\}_{i=1}^{\infty}$ is strictly increasing.

  Assume that $(x_1, x_2, \dots, x_n)$ is an arbitrary positive integer solution of the generalized Markov-Hurwitz equations (\ref{eq: generalized Markov Hurwitz equations}).
  For any mutation $\mu_j$ of $(x_1, x_2, \dots, x_n)$, the ratio number $k$ is given by
  \begin{align}
    k &= \dfrac{\mu_j(x_1, x_2, \dots, x_n)_j}{\prod_{i \neq j} x_i} = \dfrac{(\sum_{i \neq j} x_i^2) + \lambda_j \cdot \prod_{i \neq j} x_i}   {\prod_{i=1}^{n} x_i} \notag \\
      &= (n + \displaystyle \sum_{i=1}^{n} \lambda_i) - \dfrac{x_j^2 + \sum_{i \neq j} (\lambda_i \cdot \prod_{m \neq i} x_m) }    {\prod_{i=1}^{n} x_i}. 
  \end{align}
  From this formula, $k_i \leq n + \displaystyle \sum_{j=1}^{n} \lambda_j$ for any $i \in \mbN$, so the ratio number sequence $\{k_i\}_{i=1}^{\infty}$ is bounded.
  Hence, by the monotone convergence theorem, the ratio number sequence $\{k_i\}_{i=1}^{\infty}$ converges to a real number.

  Second, if the mutation sequence $\mfw$ is generic, then every coordinate $x_1, x_2, \dots, x_n$ tends to infinity as the number of mutations tends to infinity.
  Since the previous mutation can not be $\mu_j$, from \Cref{mutate to max}, $x_j$ is not the maximal one compared to other coordinates.
  Therefore, the term
  \begin{align}
    \dfrac{x_j^2 + \sum_{i \neq j} (\lambda_i \cdot \prod_{m \neq i} x_m) }    {\prod_{i=1}^{n} x_i}
  \end{align}
  tends to $0$, which implies that the ratio number sequence $\{k_i\}_{i=1}^{\infty}$ converges to $k_{\lambda} = n + \displaystyle \sum_{i=1}^{n} \lambda_i$.
\end{proof}

We now arrive at the main result of this section: the logarithmic generalized Markov-Hurwitz tree converges to the $n$-branched Euclid tree up to scalar multiplication.

\begin{theorem}\label{thm: logarithmic asymptotic}
  For any generic mutation sequence $\mfw = [w_1, w_2, \dots, w_t, \dots] \in \mbT_n$, there exists a real number $q$ such that the logarithmic generalized Markov-Hurwitz chain along $\mfw$ is asymptotic to $q$ times the corresponding classical $n$-branched Euclid chain as $t$ tends to infinity.
\end{theorem}

\begin{proof}
  By \Cref{ratio number convergence}, the ratio number sequence $\{k_i\}_{i=1}^{\infty}$ converges to $k_{\lambda} = n + \displaystyle \sum_{i=1}^{n} \lambda_i$.
  By \Cref{product form}, taking logarithms of the generalized Markov-Hurwitz chain along $\mfw$ gives a $\{\widebar{k_i}\}_{i=1}^{\infty}$-deformed Euclid chain along $\mfw$.

  Therefore, by \Cref{q times of deformed Euclid}, there exists a real number $q$ such that the logarithmic generalized Markov-Hurwitz chain along $\mfw$ is asymptotic to $q$ times the corresponding classical $n$-branched Euclid chain as $t$ tends to infinity.
\end{proof}

\section{Further Discussions}\label{sec: conj}

We exhibit a generalized Markov-Hurwitz uniqueness conjecture, which extends the classical Markov uniqueness conjecture \cite{Fro13}, and the generalized Markov uniqueness conjecture \cite{CJ25}.
\begin{conjecture}
\label{conj: CJW} If $(x_1,x_2,\dots,x_n)$ and $(x_1,x_2^{\prime},\dots,x_n^{\prime})$ are two positive integer solutions to the generalized Markov-Hurwitz equation \eqref{eq: generalized Markov Hurwitz equations} with $x_1\geq x_2\geq \dots \geq x_n$ and $x_1\geq x_2^{\prime}\geq \dots \geq x_n^{\prime}$, then $x_2=x_2^{\prime}, x_3=x_3^{\prime},\dots, x_n=x_n^{\prime}$.
\end{conjecture}
Note that \Cref{conj: CJW} is proposed for arbitrary coefficients $\lambda_i$ with $i=1,\dots,n$. 
Hence, it is equivalent to change the position of the fixed component $x_1$. That is to say, for example,
if $(x_1,x_2,\dots,x_n)$ and $(x_1^{\prime},x_2,\dots,x_n^{\prime})$ are two positive integer solutions to the generalized Markov-Hurwitz equation \eqref{eq: generalized Markov Hurwitz equations} with $x_2\geq x_1 \geq \dots \geq x_n$ and $x_2 \geq x_1^{\prime} \geq \dots \geq x_n^{\prime}$, then $x_1=x_1^{\prime}, x_3=x_3^{\prime},\dots, x_n=x_n^{\prime}$.

\begin{remarkx}
  In \cite[Conj. 8.2]{CJ25}, an analogous uniqueness conjecture for the generalized Markov equations is proposed. 
	Here, we extend the number of variables from $3$ to $n$. On the other hand, 
	\Cref{conj: CJW} also contains the special case for the classical Markov-Hurwitz equations, which, to the best of our knowledge, has not been considered.
\end{remarkx}

We close by pointing out another possible direction: one can vary the coefficients in the family studied here.
Throughout this paper, we have focused on equations of the form
\begin{align}
 \displaystyle \sum_{i=1}^n X_i^2 + \sum_{i=1}^n \lambda_i X_1\cdots \widehat{X_i} \cdots X_n = (n + \sum_{i=1}^n \lambda_i) \prod_{i=1}^n X_i.
\end{align}

One may also consider the more general family of equations:
\begin{align}
 \displaystyle \sum_{i=1}^n X_i^2 + \sum_{i=1}^n \lambda_i X_1\cdots \widehat{X_i} \cdots X_n = (a + \sum_{i=1}^n \lambda_i) \prod_{i=1}^n X_i + b,
\end{align}
where $a,b \in \mbZ_{\geq 0}$.

In this setting, the analogue of the initial solution need not be unique (cf. \Cref{thm: generate}).
Nevertheless, by following the strategy of \cite[Proposition 18]{GMR19}, one can show that there are only finitely many such initial solutions.

Moreover, the remaining arguments of the present paper extend to this more general class of equations.



\end{document}